\documentclass[reqno]{amsart}
\usepackage{amsmath, amssymb, amsthm, epsfig}
\usepackage{hyperref, latexsym}
\usepackage{url}

\def\today{\ifcase\month\or
  January\or February\or March\or April\or May\or June\or
  July\or August\or September\or October\or November\or December\fi
  \space\number\day, \number\year}

 \newtheorem{theorem}{Theorem}
 \newtheorem{lemma}[theorem]{Lemma}
 \newtheorem{proposition}[theorem]{Proposition}
 \newtheorem{corollary}[theorem]{Corollary}
 \theoremstyle{definition}

 \theoremstyle{remark}
 
 \newcommand{\mc}{\mathcal}

 \newcommand{\R}{\mathbb{R}}

 \newcommand{\ds}{\text{\rm d}s}

 \newcommand{\dt}{\text{\rm d}t}
 \newcommand{\du}{\text{\rm d}u}
 \newcommand{\dv}{\text{\rm d}v}

 \newcommand{\dx}{\text{\rm d}x}

 \newcommand{\dom}{\text{\rm d}\omega}
 \newcommand{\dvs}{\text{\rm d} v_{*}}
\newcommand{\dsig}{\text{\rm d}\sigma_n^{\alpha}}

\newcommand{\dsip}{\text{\rm d}\sigma^{\Phi}_n}
\newcommand{\dxi}{\text{\rm d}\xi_n^{b}}

\newcommand{\dnua}{\text{\rm d}\nu_{\alpha}}

\newcommand{\dnull}{\text{\rm d}\nu_{\lambda}}
 \newcommand{\dmu}{\text{\rm d}\mu(R)}
 \newcommand{\dnu}{\text{\rm d}\nu(x)}
 \newcommand{\dnup}{\text{\rm d}\nu_\Phi(x)}

 \newcommand{\fs}{f^{\star}}

\begin{document}

\title[Inequalities for the Boltzmann operator]{Convolution inequalities for the Boltzmann collision operator}
\author[R. J. Alonso, E. Carneiro and I. M. Gamba]{Ricardo J. Alonso, Emanuel Carneiro and Irene M. Gamba}
\date{\today}
\address{Dept. of Computational \&
Applied Mathematics,
Rice University,
Houston, TX 77005-1892}
\email{Ricardo.J.Alonso@rice.edu}
\address{School of Mathematics, Institute for Advanced Study, Einstein Drive, Princeton, NJ 08540}
\email{ecarneiro@math.ias.edu}
\address{Department of Mathematics, University of Texas at Austin, Austin, TX 78712-1082.}
\email{gamba@math.utexas.edu}
\allowdisplaybreaks
\numberwithin{equation}{section}

\maketitle

\begin{abstract}
We study integrability properties of a general version of the Boltzmann collision operator for hard and soft potentials in $n$-dimensions. A reformulation of the collisional integrals allows us to write the weak form of the collision operator as a weighted convolution, where the weight is given by an operator invariant under rotations. Using a symmetrization technique in $L^p$ we prove a Young's inequality for hard potentials, which is sharp for Maxwell molecules in the $L^2$ case. Further, we find a new Hardy-Littlewood-Sobolev type of inequality for Boltzmann collision integrals with soft potentials.  The same method extends to radially symmetric, non-increasing potentials that lie in some $L^{s}_{weak}$ or $L^{s}$.  The method we use resembles a Brascamp, Lieb and Luttinger approach for multilinear weighted convolution inequalities and follows a weak formulation setting.  Consequently, it is closely connected to the classical analysis of Young and Hardy-Littlewood-Sobolev inequalities. In all cases, the inequality constants are explicitly given by formulas depending on integrability conditions of the angular cross section (in the spirit of Grad cut-off).  As an additional application of the technique we also obtain estimates with exponential weights for hard potentials in both conservative and dissipative interactions.
\end{abstract}

\section{Introduction}

\subsection{Background}
The nonlinear Boltzmann equation is a classical model
for a gas at low or moderate densities. The gas in a
spatial domain \(\Omega\subseteq\R^n\), \(n\ge 2\),
is modeled by the evolution of the mass density function \(f(x,v,t)\),
\((x,v)\in \Omega\times\R^n\), modeling the probability of  finding a particle at position $x$, with  velocity \(v\)  at the time \(t\in \R\).
The transport equation
for \(f\) reads
\begin{equation}
\label{eq:boltz}
(\partial_t + v\cdot \nabla_x) f = Q(f,f)\,,
\end{equation}
where  \(Q(f,f)\) is a quadratic integral operator,
expressing the change of \(f\) due to instantaneous
binary collisions of particles. The precise form
of \(Q(f,f)\) will be introduced below,  for both conservative (elastic)~\cite{CeIlPu} and dissipative (inelastic) interactions~\cite{BrPo}. The \(Q(f,f)\) operator factorizes as the difference of two positive operators, usually denoted by the \(Q^+(f,f)(x,v,t)\) rate of gain of probability due to two pre-collisional velocities for which one of  them will take the direction $v$ and  the \(Q^-(f,f)(x,v,t)\) rate of loss of  probability due to  particles that get knocked out of the direction $v$.

In addition, these operators depend on the form of their collision kernels which model the collision frequency depending on the intramolecular potentials between interacting particles. More specifically, these kernels depend on functions of the relative speed and on the scattering angle, the latter modeled by an angular function referred as the angular cross section. In all the cases we  assume that the angular cross section is modeled by an integrable angular function on the $S^{n-1}$ sphere (this condition, in the theory of the Boltzmann equation,  is called the Grad cut-off assumption). The collisional kernels are further divided into the following classes:
  hard potentials, corresponding to unbounded forms of the relative speed, modeling stronger collision rates, and soft potentials modeling weaker collision rates, both as the relative speed is larger; and Maxwell molecule type of interactions where collisional kernels are independent of the relative speed.

\subsection{Aim of the paper}

It is the purpose of this work to investigate the $L^r$-integrability (in velocity) of the gain operator as a bilinear form
\(Q^+(f,g)(v)\) acting on probability mass densities $f(v)$ and $g(v)$, and to search for exact representation
formulas for the inequality constants and possible optimal estimates depending on the $L^p$ and $L^q$ norms of
$f$ and $g$, respectively. This provides a new detailed study of the Boltzmann equation with  harmonic and
functional analysis tools.  It is possible due to the weighted convolution nature of its collisional integral operator in weak form.  The work focus on  questions on sharp constants for estimates in such functional spaces as well as on the existence and exact description of maximizers.

In order to achieve our goals, we introduce a representation of  collisional integrals in weak form that allows us to write the gain term of the collision operator as a weighted convolution,
with the weight being given by  a suitable bilinear operator invariant under rotations, acting on the test function.

Following the initial idea developed in \cite{AC}, the main ingredient to approach the convolution estimates is an $L^p$-radial symmetrization technique, a genuinely new addition to the Boltzmann theory, used here with much more generality. The authors showed the Young's inequality for the elastic hard potential case and calculated exact constants which were proved optimal for  Maxwellian molecule type models in $L^2$. The method was applied to the strong formulation of the gain operator in the Carleman representation.

 The new representation for the weak formulation of the gain collisional integral as a bilinear form $Q^+(f,g)$ as a  symmetric weighted convolution presented here allow us, very handily, to extend  their results to collisional
 forms with soft (i.e. singular) and symmetric decaying potentials, as well as to the
 case of dissipative interactions. We also observe that the corresponding loss bilinear form  $Q^-(f,g)$ does not have the same convolution symmetry property as the gain one, and so the results we obtain reflects this discrepancy between the gain and loss part of the Boltzmann collision operator.

 More specifically, we employ still  here the $L^p$-radial symmetrization technique, but in addition we
 develop a Brascamp, Lieb and Luttinger-type approach to multilinear
weighted convolution  inequalities associated to a weak formulation of the collisional forms \cite{BL, BLL}.
In addition, like in the connections between the works of Beckner \cite{Be}, Brascamp-Lieb \cite{BL, Li} and multilinear convex  inequalities such as the sharp Young's  and the sharp Hardy-Littlewood-Sobolev's, we also  obtain Young's inequality for  hard
potentials,  and the Hardy-Littlewood-Sobolev for soft potentials where the weighted convolution structure has a singular kernel. We also show a Young's inequality for
radially symmetric, non-increasing potentials, where the weighted convolution structure contains a singular kernel.
In our case, the   invariant-under-rotations weight multiplier in the convolutional structure of   the weak formulation of the collisional operator introduces a  non-linear change of coordinates due to the angular integration in the $(n-1)$-dimensional sphere. That means our inequalities and methods are not a direct application
of those in \cite{Be, BL, BLL, Li} but rather an analog result for dissipative Boltzmann operators of collisional type. The range of convexity exponents depends on the convolution symmetry property of the associated bilinear form, and recovers the Hardy-Littlewood-Sobolev convexity relation in the gain operator $Q^+$ case, while is restricted
for the loss operator $Q^-$.

All these estimates are valid for conservative or dissipative interactions between particles,
whose restitution coefficients have absolute continuity and monotonicity properties as functions
of the impact parameter (see  \eqref{beta} below).

Finally we obtain Young's inequalities with exponential  weights for  hard potentials, both for the elastic and
strictly (dissipative) inelastic case. It is remarkable that,
 in the dissipative $e(z)<1$ case,  tails  of order $\lambda$ and decay rate $a$ are preserved (called
 stretched exponential tails), with order $0\le\lambda\le 2$.
However, we show that in the elastic case these estimates hold for classical Maxwellian tails but they
have a polynomial weighted norm in one of the components.
These last estimates are of interest for the study, for example, of $L^1$,  $L^\infty$  and tail propagation
dissipative kinetic collisional models, such as granular flows.

In all cases, the inequalities constants are given by explicit formulas depending only on certain
integrability conditions of the angular cross section (in the spirit of Grad cut-off).

In short, we hope that this work, inspired from the harmonic analysis toolbox, provides simplified methods,  extensions to the full range of exponents including soft potentials and radially non-increasing potentials, and also extensions to dissipative interactions and weighted estimates, to achieve a better qualitative understanding of all the convolution-like inequalities governing the Boltzmann theory.

\section{Main Results}

\subsection{Preliminaries} In this paper we study the integrability properties  of the Boltzmann collision operator in the case of elastic or inelastic collisions.  The gain part of this bilinear collision operator, commonly denoted by $Q^{+}(f,g)$, is non-local in both components. It is defined via duality by
\begin{equation}\label{e1}
\int_{\R^{n}}Q^{+}(f,g)(v)\psi(v)\,\dv:=\int_{\R^{n}}\int_{\R^{n}}f(v)g(v_{*})\int_{S^{n-1}}\psi(v')B(|u|,\hat{u}\cdot\omega)\,\dom\,\dvs\,\dv,
\end{equation}
where the functions $f, g, \psi \in C_0(\R^n)$ (continuous with compact support). The symbol $\hat{u}$ represents the unitary vector in the direction of $u$ ($\hat{u}=u/|u|$) and $\dom$ is the surface measure on the sphere $S^{n-1}$. The variables $v, v_*$ (pre-collision velocities), $v', v_{*}'$ (post-collision velocities) and $u$ (relative velocity) are related by
\begin{equation}\label{eq1'}
u=v-v_{*}\ \ , \ \ v'=v-\frac{\beta}{2}(u-|u|\omega) \ \ \mbox{and} \ \ v + v_{*} = v' + v_{*}'
\end{equation}
The corresponding loss part of Boltzmann collision operator, is only non-local in one of its its components, say $g$. It is defined in a strong form by
\begin{equation}\label{defQ-}
Q^{-}(f,g)(v):=\left\|b\right\|_{L^{1}(S^{n-1})}f(v)\int_{\mathbb{R}^{n}}g(v_*)\Phi(u)\dvs.
\end{equation}
We notice that, wile $Q^+$ is a bilinear operator with non-locality in both its components, $Q^-$ is non-local only in its second component $g$ and maintains locality in its first component $f$.
This difference in the nature of there non-locality plays a crucial role in the range of exponents for the 
Hardy-Littlewood-Sobolev inequalities for the $Q^-$ with singular potentials (Corollaries \ref{C9} and \ref{C10}).   
The inelastic properties of the collision operator are encoded in the positive scalar function $\beta:[0,\infty)\rightarrow[\tfrac{1}{2},1]$ defined by $\beta(z):=\tfrac{1+e(z)}{2}$, where parameter $e$ is the so-called \textit{restitution coefficient} which enjoys the following two properties that assure micro-reversibility of the interactions:
\begin{itemize}
\item [(i)] $z\mapsto e(z)$ is absolutely continuous and  non-increasing.
\item [(ii)] $z\mapsto z e(z)$ is non-decreasing.
\end{itemize}
The dependence of the restitution coefficient on the physical variables is commonly given by $z=|u|\ \sqrt{\tfrac{1-\hat{u}\cdot\omega}{2}}$, i.e. the restitution coefficient $e$, and thus $\beta$, depends only on the impact velocity
\begin{equation}\label{beta}
\beta\left(|u|\ \sqrt{\tfrac{1-\hat{u}\cdot\omega}{2}}\right)=\frac{1+e\left(|u|\ \sqrt{\tfrac{1-\hat{u}\cdot\omega}{2}}\right)}{2}\,.
\end{equation}
We point out that the model for $e$ could be more complex (for example assuming dependence on macroscopic variables like temperature), however this will not be the case in this paper.\\
The particle interaction is elastic when the parameter  $\beta=1$, and is referred as {\sl sticky} particles when $\beta=1/2$.
A complete discussion of the physical aspects of the restitution coefficient can be found in \cite{BrPo}. Standard models for the restitution coefficient, for example constant restitution coefficient
and viscoelastic hard spheres, satisfy the assumptions (i) and (ii) above. We refer the interested reader to \cite{Al}, \cite{BePu}, \cite{BGP}, \cite{GT} and \cite{MMR} for additional numerical and mathematical references that use this class of
models.

The nature of the interactions modeled by $Q^{+}$ is encoded in the kernel $B(|u|,\hat{u}\cdot\omega)$ modeled by strength of intramolecular potentials, and many physical models accept the representation (henceforth assumed)
\begin{equation*}
B(|u|,\hat{u}\cdot\omega)=|u|^{\lambda}b(\hat{u}\cdot\omega)\ \ \mbox{with}\ \ -n<\lambda.
\end{equation*}
Depending on the parameter $\lambda$ the interaction receives different names: soft-potentials when $-n<\lambda<0$, meaning that larger relative velocity corresponds to a weaker collision frequency; Maxwell molecules type of interactions when $\lambda=0$, of collision frequency independent of the relative velocity;  hard potentials when $\lambda>0$, meaning that larger relative velocity corresponds to stronger collision frequency. For the (nonnegative) angular kernel $b(\hat{u}\cdot\omega)$ we will require the Grad cut-off assumption
\begin{equation*}
\int_{S^{n-1}}b(\hat{u}\cdot\omega)\dom<\infty.
\end{equation*}
We refer to \cite{BGP} and \cite{GPV} for a detailed discussion on the inelastic collision operator.

\subsection{Description of the results} Very recently,   Alonso and Carneiro \cite{AC} revisited the $L^p$-analysis of the operator $Q^{+}$ in the elastic case (restitution coefficient $e \equiv 1$) for the case of Maxwell type of interactions and  hard potentials (i.e. $0\le\lambda\le1$) and developed, by means of radial symmetrization, a Young's inequality approach that produces sharp constants.  Their  approach  involved the analysis of the collisional integral on the gain operator in strong form by means of the Carleman integral representation.

  For our current goals, we now use a weak formulation of the collisional integral written as  weighted convolution, where the weight is a nonlinear invariant-under-rotations operator acting on the test function (see \eqref{e3} below).

From now on we work in the  setting of dissipative interactions satisfying conditions (i), (ii) and \eqref{beta} as described above.

Let $\psi$ and $\phi$ be bounded and continuous functions. Define the bilinear operator
\begin{equation}\label{P}
\mathcal{P}(\psi,\phi)(u):=\int_{S^{n-1}}\psi(u^{-})\phi(u^{+})b(\hat{u}\cdot\omega)\,\dom\,,
\end{equation}
where the symbols $u^{+}$ and $u^{-}$, commonly known as Bobylev's variables, are defined by
\begin{equation}\label{variables u}
u^{-}:=\tfrac{\beta}{2}(u-|u|\omega)\ \ \mbox{and}\ \ u^{+}:=u-u^{-}=(1-\beta)u+\tfrac{\beta}{2}(u+|u|\omega).
\end{equation}
We shall see in Section 3 that, under conditions \eqref{beta} and \eqref{variables u}, the operator (\ref{P}) has a certain invariance under the group of rotations in  $\R^n$. This operator was first  introduced by Bobylev  \cite{Bo1,Bo2} in a slightly different setting where it was shown that
\begin{equation}\label{FTQ}
\widehat{Q^{+}(f,g)}=\mathcal{P}(\hat{f},\hat{g})\, ,
\end{equation}
in the elastic Maxwell molecules case (i.e. $\lambda=0$ and $\beta\equiv 1$).
More recently, when introducing the study of the Boltzmann dissipative model for Maxwell type of interactions in \cite{BCaG00}, the authors showed that such relation  also  holds  for constant $\beta\neq1$.
For references on the use of the Fourier transform in the analysis of the elastic Boltzmann collision operator one can consult \cite{D} and \cite{KR}. In the dissipative interaction case we also refer  to
 \cite{GT} for the Fourier representation of the $Q^+$ operator.

In particular, from (\ref{e1}) and (\ref{P}) we obtain the following relation between the operators $Q^{+}$ and $\mathcal{P}$ appearing in the weak formulation of the gain operator, now written in velocity and relative velocity coordinates
\begin{align}\label{e3}
\begin{split}
\int_{\mathbb{R}^{n}}Q^{+}(f,g)(v)&\psi(v)\,\dv =\int_{\mathbb{R}^{n}}\int_{\mathbb{R}^{n}}f(v)g(v-u)\mathcal{P}(\tau_{v}\mathcal{R}\psi,1)(u)\ |u|^{\lambda}\,\du\,\dv\\
&= \int_{\mathbb{R}^{n}}\int_{\mathbb{R}^{n}}f(u+v)g(v)\mathcal{P}(1,\tau_{-v}\psi)(u)\, |u|^{\lambda}\,\du\,\dv,
\end{split}
\end{align}
where $\tau$ and $\mathcal{R}$ are the translation and reflection operators
\begin{equation*}
\tau_{v}\psi(x):=\psi(x-v)\ \ \mbox{and}\ \ \mathcal{R}\psi(x):=\psi(-x).
\end{equation*}
Representation (\ref{e3}) exhibits the nature of the weighted convolution structure of the weak formulation of the gain operator $Q^+$ and  also shows that the integrability properties of the collision operator $Q^{+}$ are closely related to those of the bilinear operator $\mathcal{P}$. A similar approach was carried out in \cite{GPV2} which relates the operator $Q^{+}$ to a slightly different angular averaging operator.

One of the advantages of working with the bilinear operator $\mc{P}$  is that a simple change of variables in (\ref{e3}) allow us to change the roles of $f$ and $g$ without essentially changing the convolution structure of the integral. This convolution symmetry will be specially important later on in the proof of the Hardy-Littlewood-Sobolev (HLS) inequality for $Q^{+}$ in the full range of exponents (Theorem \ref{HLS}). Such convolution symmetry is not present in the loss operator $Q^{-}$ as we shall see in Corollaries \ref{C9} and \ref{C10}, that the corresponding weak formulation for $Q^{-}(f,g)$ has a convolution structure only in its second component $g$, but not in its first one $f$. As a consequence, the HLS inequality for this operator will only be valid in a restricted range (Section 6.2).

\smallskip

In Section 3 we develop the $L^{p}$-analysis of the operator $\mathcal{P}$ by means of the weak  representation (\ref{e3})  using a radial symmetrization method introduced in \cite{AC} for a different representation of the gain operator in strong form by means of the Carleman representation.   This approach using (\ref{e3}), provides exact constants in some of our inequalities, both in the hard and soft potentials settings, that are sharp in some cases. In general, they are calculated depending on (explicit) integral conditions on the angular cross section.

The results of Section 3 are crucial for all results proved in the following three sections.

In Section 4 we revisit Young's inequality for hard potentials, where the new additions are the calculation of exact constants for conservative or dissipative interactions, which are proved to be sharp in the constant energy dissipation rate for Maxwell type of interactions and selected $L^p$-exponents.

To this end, consider the weighted Lebesgue spaces $L^p_{k}(\R^n)$ ($p\geq 1$, $k\geq0$) defined by the norm
\begin{equation*}
\|f\|_{L^p_k(\R^n)} = \left( \int_{\R^n} |f(v)|^p \,\left(1+ |v|^{pk}\right) \dv\right)^{1/p}.
\end{equation*}
We prove the following.
\begin{theorem}\label{Young}
Let $1 \leq p,\,q,\,r \leq \infty$ with $1/p + 1/q = 1 + 1/r$. Assume that $$B(|u|,\hat{u}\cdot\omega)=|u|^{\lambda}b(\hat{u}\cdot\omega)\,,$$
with $\lambda \geq 0$. For $\alpha \geq 0$, the bilinear operator $Q^{+}$ extends to a bounded operator\footnote{In this paper we shall prove all the inequalities for a dense subspace of smooth functions giving the desired extensions for $L^p$, $1 \leq p < \infty$. The $L^{\infty}$ bounds, when applicable, are easily treated directly just pulling out the $L^{\infty}$-norm out of the integrals.} from $L^p_{\alpha + \lambda}(\R^n) \times L^q_{\alpha + \lambda}(\R^n) \to L^r_{\alpha}(\R^n)$ via the estimate
\begin{equation}\label{YoungIn}
\left\| Q^{+}(f,g)\right\|_{L^r_{\alpha}(\R^n)} \leq C \,\|f\|_{L^p_{\alpha + \lambda}(\R^n)} \,
\|g\|_{L^q_{\alpha +\lambda}(\R^n)}.
\end{equation}
The constant $C$ is given by
\begin{align*}
\begin{split}
C=2^{\lambda + 2+ \tfrac{\alpha}{2}+ \tfrac{1}{r}}\left|S^{n-2}\right|&\left(2^{\tfrac{n}{r'}}\int^{1}_{-1}\left(\tfrac{1-s}{2}\right)^{-\tfrac{n}{2r'}}\dxi(s)
\right)^{\tfrac{r'}{q'}}\\
&\left(\int^{1}_{-1}\left[\left(\tfrac{1+s}{2}\right)+(1-\beta_{0})^{2}\left(\tfrac{1-s}{2}\right)\right]^{-\tfrac{n}{2r'}}\dxi(s)\right)^{\tfrac{r'}{p'}},
\end{split}
\end{align*}
where the measure $\xi_{n}^{b}$ on $[-1,1]$ is defined as
$\dxi(s)=b(s)(1-s^{2})^{\tfrac{n-3}{2}}\,\ds$ and $\beta_0=\beta(0)$. In the case $(p,q,r) = (1,1,1)$ the constant $C$ is to be understood as
\begin{equation*}
 C = 2^{\lambda + 3 + \tfrac{\alpha}{2}}\left|S^{n-2}\right|\int^{1}_{-1}\dxi(s).
\end{equation*}

\end{theorem}

In Section 5 we prove a Hardy-Littlewood-Sobolev-type inequality for the collision operator in the case of soft potentials, as follows:
\begin{theorem}\label{HLS}
Let $1< p,\ q,\ r< \infty$ with $-n<\lambda<0$ and $1/p+1/q=1+\lambda/n+1/r$.  For the kernel
$$B(|u|,\hat{u}\cdot\omega)=|u|^{\lambda}\ b(\hat{u}\cdot\omega),$$
the bilinear operator $Q^{+}$ extends to a bounded operator from $L^p(\mathbb{R}^{n}) \times L^q(\mathbb{R}^{n}) \to L^r(\mathbb{R}^{n})$ via the estimate
\begin{equation}\label{e18}
\left\| Q^{+}(f,g)\right\|_{L^r(\mathbb{R}^{n})} \leq C \,\|f\|_{L^{p}(\mathbb{R}^{n})} \, \|g\|_{L^q(\mathbb{R}^{n})}.
\end{equation}
The constant  $C$ is explicit in {\rm(\ref{boundb1*})}, {\rm(\ref{boundb2*})} and {\rm(\ref{boundb3})}.
\end{theorem}

The constants we obtain for the two inequalities above are explicit, but generally not sharp. Only in the cases $\alpha = \lambda = 0$, $(p,q,r) = (2,1,2)$ and $(p,q,r) = (1,2,2)$ we find the sharp constant for the Young's inequality (\ref{YoungIn}) (see Corollary \ref{MMsharp}). In fact, the quest for the sharp forms of these inequalities in the other cases, which could be seen as analogues of the remarkable works of Beckner \cite{Be}, Brascamp-Lieb \cite{BL} and Lieb \cite{Li}, seems to be a very difficult problem in harmonic analysis.

Finally, in Section 6, we provide further applications of the methods described here. We first present in Section 6.1 a description of the estimates for collision kernels with radial non-increasing potentials that leads to multiple Young's inequalities in Corollaries~\ref{C7} and~\ref{C8}. 
In Section 6.2, we study the corresponding inequalities for  the loss operator $Q^-$ in Corollaries~\ref{C9} and~\ref{C10}, and see that the lack of convolution symmetry restricts the range of exponents.

In Section 6.3 we prove Young-type estimates with exponential  weights for  hard potentials both for the elastic and strictly (dissipative) inelastic case.  We show in Proposition~\ref{IMW} that in the elastic case the weighted estimates with classical Maxwellian tails hold but have a polynomial weighted norm in one of the components.
However, things get better in the strict dissipative $e(z)<1$ case, as shown in Theorem~\ref{ISW}, where tails  of order $\lambda$ and decay rate $a$ are preserved (called stretched exponential tails), with order $0\le\lambda\le 2$.
 These are important tools in the study of propagation of moments \cite{BGP} and $L^{1}-L^{\infty}$-exponentially weighted comparison principles \cite{GPV}.

\subsection{Related literature}  Young-type inequalities in Theorem~\ref{Young} for the collisional integrals,
in the case of Maxwell type and hard potentials, reveal the convolution nature of the operator $Q^{+}$. In the elastic case,
this observation was first introduced by Gustafsson \cite{G} for  an assumed angular cross section function with pointwise
 cut-off from away from $zero$ (grazing collisions) and $\pi$ (head-on collisions). Later, Mouhot and Villani \cite[Theorem 2.1]{MV}  revisited the work of Gustafsson under the same conditions to obtain a
different control of the constants (observe however that the constant in their Young's inequality blows up at the endpoints, due to the pointwise cut-off assumption). Also, standard integrability of the angular cross section is required away from the cut-off. Later, Duduchava-Kirsch-Rjasanow \cite{DKR} used  elementary techniques to obtain a simpler proof of the control
of the $L^p$-norm for the $Q^+$ operator in three dimensions by means of  $L^{1}$  and $L^{p}$ norms where the constants in our present
work match precisely the ones obtained in \cite{DKR} for the case $(p,q,r) = (p,1,p)$.
 About the same time,  Gamba, Panferov and Villani, \cite[Lemma 4.1]{GPV2}, presented related estimates that use dual integrability
of the functions, namely, $f$ and $g$ belonging to $L^{1}\cap L^{p}$. This stronger assumption presented in \cite{GPV2} allowed them
to remove the restriction of pointwise cut-offs and use only the integrability of the cross section. Related work on $L^p_k$ estimates for the elastic or the dissipative collisional integrals was also done by Gamba-Panferov-Villani \cite{GPV},  Bobylev-Gamba-Panferov \cite{BGP}, and Mischler-Mouhot-Ricard \cite{MMR}.

 In particular, the following observation is worth stressing: some techniques valid in the elastic case are not available in the inelastic (dissipative) one. For instance, defining the angular kernel in a half spherical domain is not possible due to the lack of symmetry, and in particular, such lack of symmetry forces
 a dual integrability assumption in the  derivation of $L^p$ estimates done in \cite[Lemma 4.1]{GPV2}.
In fact, dual integrability assumption is adequate  for the \textit{quadratic} operator, and also  in some cases for bilinear estimates, such as when studying propagation of derivatives for the conservative case.
However, it is not always the case that such stronger assumptions hold, an so one needs to study a purely bilinear estimate.  This is precisely the case for comparison techniques, where a function $g$, which in principle is unrelated to a solution $f$, is ``compared'' with the solution and so  the bilinear form $Q(g,f)$ appears.  In general, there is no reason on why this function $g$ may  have dual integrability, and so, the previous technique will not apply.
Finally, we mention that the authors in  \cite{BGP} have already used the angular averaging mechanism  occurring for $\mathcal{P}$ (to be introduced in the next section) in order to obtain sharp moment decay formulas for the gain operator, both in the case of elastic and inelastic (with constant restitution coefficient) hard potentials.  These results lead to the study of the regularity and the aforementioned Gaussian propagation for the solutions of the Boltzmann equation using comparison techniques \cite{AG08} and \cite{GPV}.

In the soft potentials case, Theorem~\ref{HLS} also reinforces the convolution character of $Q^{+}(f,g)$.
It is however important to notice that the weak formulation of the collisional integral as a weighted convolution
(representation (\ref{e3})) is crucial to attain the results by estimating multilinear integrals by convex type estimates, maximized in their symmetrization and the connections to Brascamp-Lieb-Luttinger inequalities \cite{BLL, BL}. It is also worth mentioning that the classical Hardy-Littlewood-Sobolev was used in connection to  Cancellation Lemmas of  Alexandre-Desvillettes-Villani-Wennberg \cite{ADVW} and Villani \cite{Vi} who constructed estimates with singular kernels and non-integrable cross sections, assuming $L^p_k$ integrability to control the $L^1$-norm of the total collisional operator. Our work is not related to theirs, but rather extends to any radial non-increasing potential,  showing that the collisional integral satisfies a Hardy-Littlewood-Sobolev-type inequality both for $Q^+$ and $Q^-$ independently (see also Corollaries~\ref{C9} and \ref{C10}), providing exact representations for the constants, and assuming only $L^p$-integrability. In this sense it is the first time where such connection was made.

Finally, recent main applications of this work are mentioned below. Alonso and Gamba \cite{AG}
used  this result to obtain classical solutions and $L^p$-stability , $1\le  p\le\infty$,  for the
Cauchy problem associated to the Boltzmann equation for soft potentials  and singular radially
symmetric decreasing potentials, with integrable cross section, for initial data near vacuum or
large data near local Maxwellian distribution. In addition, Alonso and Lods \cite{Alon-Lods} have
shown the pointwise control to inelastic  homogeneous cooling solutions for the inelastic Boltzmann
equation for hard spheres and the conjectured Haff Law by means of a corresponding Young-type estimate
for the inelastic Boltzmann collision operator with stretched exponential weights.

\medskip


\section{Radial Symmetrization and the Operator $\mc{P}$}
Let $G = SO(n)$ be the group of rotations of $\R^n$ (orthonormal transformations of determinant $1$), in which we will use the variable $R$ to designate a generic rotation. We assume that the Haar measure $\textrm{d} \mu$ of this compact topological group is normalized so that
\begin{equation*}\label{S2.14}
\int_{G} \dmu = 1.
\end{equation*}
Let $f \in L^p(\R^n)$, $p\geq 1$. We define the radial symmetrization $\fs_p$ by
\begin{equation}\label{S2.15}
\fs_p(x) = \left(\int_{G} |f(Rx)|^p \,\dmu\right)^{\tfrac{1}{p}}\,, \ \ \textrm{if} \ \ 1\leq p< \infty.
\end{equation}
and
\begin{equation}\label{S2.15.1}
\fs_{\infty}(x) = \textrm{ess sup}_{|y| =|x|} |f(y)|
\end{equation}
where the essential sup in (\ref{S2.15.1}) is taken over the sphere of radius $|x|$ with respect to the surface measure over this sphere. This radial rearrangement $\fs_p$ defined in (\ref{S2.15})-(\ref{S2.15.1}) can be seen as an $L^p$-average of $f$ over all the rotations $R \in G$ and it satisfies the following properties:
\begin{itemize}
 \item[(i)] $\fs_p$ is radial.
\item[(ii)] If $f$ is continuous (or compactly supported) then $\fs_p$ is also continuous (or compactly supported).
\item[(iii)] If $g$ is a radial function then $(fg)^{\star}_p(x) = \fs_p(x)g(x)$.
\item[(iv)] Let $\textrm{d} \nu$ be a rotationally invariant measure on $\R^n$. Then
$$\int_{\R^n} |f(x)|^p \, \dnu = \int_{\R^n} |\fs_p(x)|^p \, \dnu.$$
In particular,
\begin{equation*}\label{S2.15.2}
\|f\|_{L^p(\R^n)} = \|\fs_p\|_{L^p(\R^n)}.
\end{equation*}
\end{itemize}
We noticed that this is not the classical rearrangement by monotonicity of level sets associated to $f$, stills has the property that preserves the $L^p$-norm, and so we are driven to prove a convex inequality type of result  that can be viewed as an analog to a Brascamp Lieb and Luttinger type inequalities, where  multilinear convolution type integrals are maximized on their (classical) rearrangements
and controlled by multiple Minkowski type of convexity estimates \cite{BLL,BL}.

The following result, our first in this manuscript,  is a redo of \cite[Lemma 4]{AC}. We
 produce here a different proof of that in \cite{AC} (which use a Carleman integral type representation), and present a new approach that exhibits the similarity to those arguments of classical rearrangement inequalities such as those in  \cite{BLL,BL}. These new approach is
 crucial to perform the extensions to soft potentials through Hardy-Littlewood-Sobolev and Lieb \cite{LiLo} type of inequalities.

\begin{lemma}\label{sym lemma}
Let $f, g, \psi \in C_0(\mathbb{R}^n)$ and $1/p \,+ \,1/q \,+\, 1/r = 1$, with $1\leq p,q,r \leq \infty$. Then
\begin{equation*}\label{e7}
\left|\int_{\mathbb{R}^{n}}\mathcal{P}(f,g)(u)\psi(u)\,\du \right|\leq \int_{\mathbb{R}^{n}}\mathcal{P}(f^{\star}_p,g^{\star}_q)(u)\psi^{\star}_r(u)\, \du.
\end{equation*}
\end{lemma}
\begin{proof}
From (\ref{beta}), (\ref{P}) and (\ref{variables u}) we observe that for any rotation $R$ one has
\begin{equation*}
\mathcal{P}(f,g)(Ru)=\mathcal{P}(f\circ R,g\circ R)(u).
\end{equation*}
Therefore,
\begin{align}
\begin{split}\label{Sym}
\Bigl|\int_{\mathbb{R}^{n}}\mathcal{P}(f,g)(u)&\,\psi(u)\, \du\Bigr| =\Bigl|\int_{\mathbb{R}^{n}}\mathcal{P}(f,g)(Ru)\,\psi(Ru)\, \du\Bigr|\\
&=\Bigl|\int_{\mathbb{R}^{n}}\mathcal{P}(f\circ R,g\circ R)(u)\,\psi(Ru)\, \du\Bigr|\\
&\leq\int_{\mathbb{R}^{n}}\int_{S^{n-1}}|f(Ru^{-})|\,|g(Ru^{+})| \,|\psi(Ru)|b(\hat{u}\cdot\omega)\,\dom\,\du.
\end{split}
\end{align}
Note that the left hand side of (\ref{Sym}) is independent of $R$.  Thus, an integration over the group $G=SO(n)$ leads to
\begin{align}\label{e9}
\begin{split}
\Bigl|\int_{\mathbb{R}^{n}}\mathcal{P}&(f,g)(u)\psi(u)\, \du\Bigr|\\
&\leq\int_{\mathbb{R}^{n}}\int_{S^{n-1}}\left(\int_{G}|f(Ru^{-})|\, |g(Ru^{+})|\, |\psi(Ru)|\,\dmu\right) b(\hat{u}\cdot\omega)\,\dom \,\du.
\end{split}
\end{align}
An application of H\"{o}lder's inequality with exponents $p$, $q$ and $r$ yields
\begin{equation*}
\int_{G}|f(Ru^{-})|\,|g(Ru^{+})|\,|\psi(Ru)|\,\dmu\leq f^{\star}_p(u^{-})\,g^{\star}_q(u^{+})\,\psi^{\star}_r(u),
\end{equation*}
which together with equation (\ref{e9}) proves the lemma.
\end{proof}

Lemma \ref{sym lemma} shows that $L^{p}$-estimates for the operator $\mathcal{P}$ will follow by considering radial functions. If $f:\mathbb{R}^{n}\rightarrow\mathbb{R}$ is radial, we define the function $\tilde{f}:\R^{+}\rightarrow\mathbb{R}$ by
\begin{equation*}
f(x)=\tilde{f}(|x|).
\end{equation*}
In addition, for any $p\geq1$ and $\alpha\in\mathbb{R}$ we have
\begin{equation}\label{Lp_alpha}
\int_{\mathbb{R}^{n}}f(x)^{p}\;|x|^{\alpha}\,\dx=\left|S^{n-1}\right|\int^{\infty}_{0}\tilde{f}(t)^{p}\;t^{n-1+\alpha}\,\dt.
\end{equation}
Hence, if we define the measure $\nu_{\alpha}$ on $\R^n$ by
\begin{equation*}
\dnua(x)=|x|^{\alpha}\dx\,,
\end{equation*}
and the measure $\sigma^{\alpha}_{n}$ on $\R^{+}$ by
\begin{equation*}
\dsig(t)=t^{n-1+\alpha}\dt\,,
\end{equation*}
equation (\ref{Lp_alpha}) translates to
\begin{equation}\label{e10}
||f||_{L^{p}(\mathbb{R}^{n},\, \dnua)}=\left|S^{n-1}\right|^{\tfrac{1}{p}}\;||\tilde{f}||_{L^{p}(\mathbb{R}^{+},\, \dsig)}.
\end{equation}
In the following computation we show how the operator $\mathcal{P}$ simplifies to a $1$-dimensional operator when applied to radial functions.  If $f$ and $g$ are radial, then
\begin{align}\label{e11}
\begin{split}
\mathcal{P}(f,g)(u)&=\int_{S^{n-1}}\tilde{f}\left(|u^{-}|\right)\tilde{g}\left(|u^{+}|\right)\, b(\hat{u}\cdot\omega)\, \dom\\
&=\int_{S^{n-1}}\tilde{f}\bigl(a_1(|u|, \hat{u}\cdot\omega)\bigr)\tilde{g}\bigl(a_2(|u|, \hat{u}\cdot\omega)\bigr)\, b(\hat{u}\cdot\omega)\,\dom\\
&=\left|S^{n-2}\right|\int^{1}_{-1}\tilde{f}\bigl(a_1(|u|, s)\bigr)\tilde{g}\bigl(a_2(|u|, s)\bigr) b(s)\,(1-s^{2})^{\tfrac{n-3}{2}}\,\ds.
\end{split}
\end{align}
The functions $a_1$ and $a_2$ are defined on $\mathbb{R}^{+}\times[-1,1]\rightarrow\mathbb{R}^{+}$ by
\begin{equation}\label{e11.5}
a_1(x,s)=\beta\ x\left(\tfrac{1-s}{2}\right)^{1/2}\ \ \mbox{and}\ \ a_2(x,s)=x\left[\left(\tfrac{1+s}{2}\right)+(1-\beta)^{2}\left(\tfrac{1-s}{2}\right)\right]^{1/2}.
\end{equation}

We conclude from (\ref{e11}) that
\begin{equation}\label{e12}
\widetilde{\mathcal{P}(f,g)}(x)=\left|S^{n-2}\right|\int^{1}_{-1}\tilde{f}\left(a_1(x, s)\right)\,\tilde{g}\left(a_2(x, s)\right)\, \dxi(s)\,,
\end{equation}
where the measure $\xi_{n}^{b}$ on $[-1,1]$ is defined as
\begin{equation*}
\dxi(s)=b(s)(1-s^{2})^{\tfrac{n-3}{2}}\,\ds\,.
\end{equation*}
In virtue of equation (\ref{e12}) we define the following bilinear operator for any two bounded and continuous functions $f,g :\R^{+} \to \R$,
\begin{equation}\label{e12.5}
\mathcal{B}(f,g)(x):=\int^{1}_{-1}f\left(a_1(x, s)\right)\,g\left(a_2(x, s)\right)\dxi(s).
\end{equation}
{\bf Remark.}
It is worth to notice that in the case of constant parameter $\beta$ (which includes elastic interactions) the functions
$a_1$ and $a_2$ of the variable interactions are actually functions of the form $a_1=x \alpha_1(s)$ and $a_2=x\alpha_2(s)$;
that is  $a_1$ and $a_2$ are first order homogeneity in
 their radial part and their angular part is a positive,  bounded by unity function of the angular parametrization $s$. This observation is at the heart of the analysis made in \cite{BCG-09} which shows compactness properties of the spectral structure associated to the  bilinear form \eqref{e12.5}.

For the operator in \eqref{e12.5} we have the following bound.
\begin{lemma}\label{L1}
Let $1\leq p,q,r\leq\infty$ with $1/p + 1/q = 1/r$. For $f \in L^p(\mathbb{R}^{+},\dsig)$ and $g \in L^q(\mathbb{R}^{+},\dsig)$ we have
\begin{equation}\label{e13}
\left\|\mathcal{B}(f,g)\right\|_{L^r(\mathbb{R}^{+},\, \dsig)}\leq C\, \left\|f\right\|_{L^p(\mathbb{R}^{+},\, \dsig)}\left\|g\right\|_{L^q(\mathbb{R}^{+},\,\dsig)},
\end{equation}
where the constant $C$ is given in {\rm (\ref{const B})} below. In the case of constant restitution coefficient $e$, corresponding to a constant parameter $\beta=(1+e)/2$, one can show that
\begin{equation}\label{sharp const}
C(n,\alpha,p,q,b, \beta)=\beta^{-\tfrac{n+\alpha}{p}}\int^{1}_{-1}\left(\tfrac{1-s}{2}\right)^{-\tfrac{n+\alpha}{2p}}\left[\left(\tfrac{1+s}{2}\right)+(1-\beta)^{2}\left(\tfrac{1-s}{2}\right)\right]^{-\tfrac{n+\alpha}{2q}}\dxi(s)
\end{equation}
is sharp.
\end{lemma}

\begin{proof}
Using Minkowski's inequality and H\"{o}lder's inequality with exponents $p/r$ and $q/r$ we obtain
\begin{align*}\label{e14}
\begin{split}
\bigl\|\mathcal{B}(f&,g)\bigr\|_{L^r(\mathbb{R}^{+},\,\dsig)}\leq\int^{1}_{-1}\left(\int^{\infty}_{0}\left|f(a_1(x,s))\right|^{r}\left|g(a_2(x,s))\right|^{r}\dsig(x)\right)^{\tfrac{1}{r}}\dxi(s)\\
&\leq\int^{1}_{-1}\left(\int^{\infty}_{0}\left|f(a_1(x,s))\right|^{p}\dsig(x)\right)^{\tfrac{1}{p}}\left(\int^{\infty}_{0}\left|g(a_2(x,s))\right|^{q}\dsig(x)\right)^{\tfrac{1}{q}}\dxi(s).
\end{split}
\end{align*}
Since the function $z\rightarrow z e(z)$ is non-decreasing, the change of variables $y=a_1(x,s)$ is valid for any fixed $s\in[-1,1)$, and its inverse Jacobian satisfies
\begin{equation}\label{Jacob 1}
\left|\frac{\textrm{d}a_1}{\dx}\right|\geq \frac{1}{2}\left(\tfrac{1-s}{2}\right)^{\tfrac{1}{2}}.
\end{equation}
Moreover, using the fact that $\beta \geq 1/2$, we arrive at
\begin{equation}\label{eap1}
\left(\int^{\infty}_{0}\left|f(a_1(x,s))\right|^{p}\dsig(x)\right)^{\tfrac{1}{p}}\leq2^{\tfrac{n+\alpha}{p}}\left(\tfrac{1-s}{2}\right)^{-\tfrac{n+\alpha}{2p}}\left\|f\right\|_{L^p(\mathbb{R}^{+},\, \dsig)}.
\end{equation}
Using a similar analysis for the change of variables $y=a_2(x,s)$, exploiting the fact that $\beta$ is non-increasing, we obtain
\begin{equation}\label{Jacob 2}
\left|\frac{\textrm{d}a_2}{\dx}\right|\geq \left[\left(\tfrac{1+s}{2}\right)+(1-\beta_0)^{2}\left(\tfrac{1-s}{2}\right)\right]^{\tfrac{1}{2}}\,,
\end{equation}
where $\beta_{0}=\beta(0)$.  We then arrive at
\begin{equation*}
\left(\int^{\infty}_{0}\left|g(a_2(x,s))\right|^{q}\dsig(x)\right)^{\tfrac{1}{q}}\leq\left[\left(\tfrac{1+s}{2}\right)+(1-\beta_0)^{2}\left(\tfrac{1-s}{2}\right)\right]^{-\tfrac{n+\alpha}{2q}}\left\|g\right\|_{L^q(\mathbb{R}^{+},\, \dsig)},
\end{equation*}
This gives (\ref{e13}) with constant
\begin{equation}\label{const B}
C=2^{\tfrac{n+\alpha}{p}}\int^{1}_{-1}\left(\tfrac{1-s}{2}\right)^{-\tfrac{n+\alpha}{2p}}\left[\left(\tfrac{1+s}{2}\right)+(1-\beta_0)^{2}\left(\tfrac{1-s}{2}\right)\right]^{-\tfrac{n+\alpha}{2q}}\dxi(s)\,.
\end{equation}
In the case of constant $\beta$, the Jacobians (\ref{Jacob 1}) and (\ref{Jacob 2}) can be explicitly computed and the proposed change of variables leads to the constant (\ref{sharp const}). To prove that the constant (\ref{sharp const}) is the best possible in this case, one can consider the sequences $\{f_\epsilon\}$ and $\{g_{\epsilon}\}$ with $\epsilon >0$ defined by
\begin{equation*}
f_{\epsilon}(x) =
\left\{
\begin{array}{cl}
\epsilon^{1/p}\, x^{-(n+\alpha - \epsilon)/p} & \textrm{for}  \ \ 0<x<1\,,\\
0 & \textrm{otherwise}.
\end{array}
\right.
\end{equation*}
and
\begin{equation*}
g_{\epsilon}(x) =
\left\{
\begin{array}{cl}
\epsilon^{1/q}\, x^{-(n+\alpha - \epsilon)/q} & \textrm{for}  \ \ 0<x<1\,,\\
0 & \textrm{otherwise}.
\end{array}
\right.
\end{equation*}
Clearly,
\begin{equation*}
\left\|f_\epsilon\right\|_{L^p(\mathbb{R}^{+},\, \dsig)} = \left\|g_\epsilon\right\|_{L^q(\mathbb{R}^{+},\, \dsig)}=1\,,
\end{equation*}
and one can check that
\begin{equation*}
\left\|\mathcal{B}(f_{\epsilon},g_{\epsilon})\right\|_{L^r(\mathbb{R}^{+},\, \dsig)} \rightarrow C\,,
\end{equation*}
as $\epsilon \to 0$, where $C$ is the constant defined in (\ref{sharp const}). The detailed argument is outlined in \cite{AC}, in the case $\beta = 1$.
\end{proof}

From Lemma \ref{sym lemma} we have
\begin{equation*}
\|\mathcal{P}(f,g)\|_{L^r(\mathbb{R}^n,\,\dnua)}\leq\|\mathcal{P}(f^{\star}_{p},g^{\star}_{q})\|_{L^r(\mathbb{R}^n,\, \dnua)}\,,
\end{equation*}
where $1/p+1/q=1/r$. Using equations (\ref{e10}), (\ref{e12}) and Lemma \ref{L1} we obtain
\begin{align}\label{e15}
\begin{split}
\|\mathcal{P}(f^{\star}_{p},g^{\star}_{q})&\|_{L^r(\mathbb{R}^n,\, \dnua)}=\left|S^{n-1}\right|^{\tfrac{1}{r}}\ \left\|\widetilde{\mathcal{P}(f^{\star}_{p},g^{\star}_{q})}\right\|_{L^r(\mathbb{R}^+,\, \dsig)}\\
&=\left|S^{n-1}\right|^{\tfrac{1}{r}}\ \left|S^{n-2}\right|\ \|\mathcal{B}(\tilde{f^{\star}_p},\tilde{g^{\star}_q})\|_{L^r(\mathbb{R}^+,\, \dsig)}\\
&\leq C\ \left|S^{n-1}\right|^{\tfrac{1}{r}}\ \left|S^{n-2}\right|\ \|\tilde{f^{\star}_p}\|_{L^p(\mathbb{R}^+,\, \dsig)}\ \|\tilde{g^{\star}_q}\|_{L^q(\mathbb{R}^+,\, \dsig)}\\
&=C\ \left|S^{n-2}\right|\ \|f\|_{L^p(\mathbb{R}^n,\, \dnua)}\ \|g\|_{L^q(\mathbb{R}^n,\, \dnua)}\,,
\end{split}
\end{align}
and thus we have proved the following result.
\begin{theorem}\label{T3}
Let $1 \leq p,q,r \leq \infty$ with $1/p + 1/q = 1/r$, and $\alpha \in\mathbb{R}$. The bilinear operator $\mathcal{P}$ extends to a bounded operator from $L^p(\mathbb{R}^n, \dnua) \times L^q(\mathbb{R}^n, \dnua)$ to $L^r(\mathbb{R}^n, \dnua)$ via the estimate
\begin{equation*}
\left\|\mathcal{P}(f,g)\right\|_{L^r(\mathbb{R}^n,\, \dnua)}\leq \, C \, \left\|f\right\|_{L^p(\mathbb{R}^n,\, \dnua)}\left\|g\right\|_{L^q(\mathbb{R}^n,\, \dnua)}.
\end{equation*}
Moreover, in the case of constant restitution coefficient $e$, the constant
\begin{equation*}
C = \left|S^{n-2}\right|\beta^{-\tfrac{n+\alpha}{p}}\int^{1}_{-1}\left(\tfrac{1-s}{2}\right)^{-\tfrac{n+\alpha}{2p}}\left[\left(\tfrac{1+s}{2}\right)+(1-\beta)^{2}\left(\tfrac{1-s}{2}\right)\right]^{-\tfrac{n+\alpha}{2q}}\dxi(s)
\end{equation*}
is sharp.
\end{theorem}
An application of Theorem \ref{T3} with Fourier transform methods provides sharp estimates for the $(2,1,2)$ and $(1,2,2)$ Young's inequalities in the case of Maxwell molecules and constant parameter $\beta$ (they will be treated with more generality in the next section). These are the only cases where we are able to explicitly find the sharp constant and exhibit a family of maximizers for the Young's inequality.

\begin{corollary}\label{MMsharp}
 Let $f\in L^{1}(\mathbb{R}^{n})$ and $g\in L^{2}(\mathbb{R}^{n})$. Then
\begin{equation}
\left\|Q^{+}(f,g)\right\|_{L^{2}(\mathbb{R}^{n})} \leq C_0\,\|f\|_{L^{1}(\mathbb{R}^{n})}\,\|g\|_{L^{2}(\mathbb{R}^{n})}, \label{sharpYoung23}
\end{equation}
with the sharp constant given by
\begin{equation*}
C_0=\left|S^{n-2}\right|\int^{1}_{-1}\left[\left(\tfrac{1+s}{2}\right)+(1-\beta)^{2}\left(\tfrac{1-s}{2}\right)\right]^{-\tfrac{n}{4}}\dxi(s)\,.
\end{equation*}
Similarly, for $f\in L^{2}(\mathbb{R}^{n})$ and $g\in L^{1}(\mathbb{R}^{n})$ we have
\begin{equation}\label{sharpYoung2}
\left\|Q^{+}(f,g)\right\|_{L^{2}(\mathbb{R}^{n})}\leq C_1\,\|f\|_{L^{2}(\mathbb{R}^{n})}\,\|g\|_{L^{1}(\mathbb{R}^{n})},
\end{equation}
with the sharp constant given by
\begin{equation*}
C_1= \left|S^{n-2}\right|\beta^{-\tfrac{n}{2}}\int^{1}_{-1}\left(\tfrac{1-s}{2}\right)^{-\tfrac{n}{4}}\, \dxi(s).
\end{equation*}
\end{corollary}
\begin{proof}
To prove (\ref{sharpYoung23}) observe that
\begin{align}
\begin{split}
\left\|Q^{+}(f,g)\right\|_{L^{2}(\mathbb{R}^{n})}&=\left\|\widehat{Q^{+}(f,g)}\right\|_{L^{2}(\mathbb{R}^{n})}
=\left\|\mathcal{P}(\hat{f},\hat{g})\right\|_{L^{2}(\mathbb{R}^{n})}\\
& \leq C_0\,\|\hat{f}\|_{L^{\infty}(\mathbb{R}^{n})}\,\|\hat{g}\|_{L^{2}(\mathbb{R}^{n})}\leq C_0\,\|f\|_{L^{1}(\mathbb{R}^{n})}\,\|g\|_{L^{2}(\mathbb{R}^{n})}, \label{sharpYoung}
\end{split}
\end{align}
with the constant $C_0$ coming directly from Theorem \ref{T3}.
To guarantee that $C_0$ is indeed the sharp constant in the inequality (\ref{sharpYoung}) we need approximating sequences $\widetilde{\hat{f}}_{\epsilon}$ and $\widetilde{\hat{g}}_{\epsilon}$
slightly different from those presented in the end of the proof of Lemma \ref{L1}, since we would like to impose the additional constraint $f \geq 0$ to have $\|\hat{f}\|_{L^{\infty}(\mathbb{R}^{n})}=\|f\|_{L^{1}(\mathbb{R}^{n})}$. Heuristically, this can be done by considering $f = \delta(x)$ the Dirac delta and so $\hat{f} \equiv 1$. In practice we should choose $f_{\epsilon}$ a Gaussian approximation of the identity by putting
\begin{equation*}
\widetilde{\hat{f}}_{\epsilon}(x) = e^{-\pi \epsilon^2 x^2}\,,
\end{equation*}
and
\begin{equation*}
\widetilde{\hat{g}}_{\epsilon}(x) =
\left\{
\begin{array}{cl}
\epsilon^{1/2}\, x^{-(n- \epsilon)/2} & \textrm{for}  \ \ 0<x<1\,,\\
0 & \textrm{otherwise}.
\end{array}
\right.
\end{equation*}
A similar consideration applies to the inequality (\ref{sharpYoung2}).
\end{proof}

\section{Young's Type Inequality for Hard Potentials}

The goal of this section is to prove Theorem \ref{Young}. First we treat the case $\alpha = \lambda = 0$. The main idea is to use the relation (\ref{e3}) that establishes a connection between the operators $Q^{+}$ and $\mc{P}$, and then use the $L^p$-knowledge of the operator $\mc{P}$ from the previous section. In what follows we assume that all the functions involved are non-negative (since when working with $L^p$-norms one can always use the modulus of a function). From (\ref{e3}) we have
\begin{equation}\label{Young1.1}
I:=\int_{\mathbb{R}^{n}}Q^{+}(f,g)(v)\psi(v)\,\dv=\int_{\mathbb{R}^{n}}\int_{\mathbb{R}^{n}}f(v)g(v-u)\mathcal{P}(\tau_{v}\mathcal{R}\psi,1)(u)\,\du\,\dv.
\end{equation}
Suppose first that $(p,q,r) \neq (1,1,1), (1, \infty, \infty), (\infty,1,\infty)$. The exponents $p,q,r$ in Theorem \ref{Young} satisfy $1/p' + 1/q' + 1/r = 1$, and thus we can regroup the terms conveniently and use H\"{o}lder's inequality
\begin{align}\label{Holder1}
\begin{split}
I=\int_{\mathbb{R}^{n}}\int_{\mathbb{R}^n}\Bigl(f(v)^{\tfrac{p}{r}}&g(v-u)^{\tfrac{q}{r}}\Bigr) \,\Bigl(f(v)^{\tfrac{p}{q'}} \mathcal{P}(\tau_{v}\mathcal{R}\psi,1)(u)^{\tfrac{r'}{q'}}\Bigr)\\& \Bigl(g(v-u)^{\tfrac{q}{p'}}\mathcal{P}(\tau_{v}\mathcal{R}\psi,1)(u)^{\tfrac{r'}{p'}}\Bigr) \du\,\dv\ \leq \ I_1\,I_2\,I_3,
\end{split}
\end{align}
where
\begin{align}\label{add0}
\begin{split}
&I_1:=\left(\int_{\mathbb{R}^n} \int_{\mathbb{R}^n} f(v)^{p} g(v-u)^{q}\, \du\,\dv\right)^{\tfrac{1}{r}}\\
&I_2:=\left(\int_{\mathbb{R}^n} \int_{\mathbb{R}^n} f(v)^{p} \mathcal{P}(\tau_{v}\mathcal{R}\psi,1)(u)^{r'}\du\,\dv\right)^{\tfrac{1}{q'}}\\
&I_3:=\left(\int_{\mathbb{R}^n}\int_{\mathbb{R}^n}g(v-u)^{q}\mathcal{P}(\tau_{v}\mathcal{R}\psi,1)(u)^{r'}\du\,\dv\right)^{\tfrac{1}{p'}}\\&\ \ \ \ \ \ \ \ \ \ \ \ \ \ \ \ \ \ \ \ \ \ \ \ \ \ \ \ =\left(\int_{\mathbb{R}^n}\int_{\mathbb{R}^n}g(v)^{q}\mathcal{P}(1,\tau_{-v}\psi)(u)^{r'}\du\,\dv\right)^{\tfrac{1}{p'}}.
\end{split}
\end{align}
Recall that $\tau$ and $\mathcal{R}$ are unitary operators in the $L^{p}$ spaces, thus, from (\ref{Holder1}) and Theorem \ref{T3} we obtain
\begin{equation*}
I\leq C \,\|f\|_{L^{p}(\mathbb{R}^{n})} \, \|g\|_{L^q(\mathbb{R}^{n})}\, \|\psi\|_{L^{r'}(\mathbb{R}^{n})},
\end{equation*}
with constant given by
\begin{align}\label{e17}
\begin{split}
C=\left|S^{n-2}\right|&\left(2^{\tfrac{n}{r'}}\int^{1}_{-1}\left(\tfrac{1-s}{2}\right)^{-\tfrac{n}{2r'}}\dxi(s)
\right)^{\tfrac{r'}{q'}}\\
&\left(\int^{1}_{-1}\left[\left(\tfrac{1+s}{2}\right)+(1-\beta_{0})^{2}\left(\tfrac{1-s}{2}\right)\right]^{-\tfrac{n}{2r'}}\dxi(s)\right)^{\tfrac{r'}{p'}},
\end{split}
\end{align}
which concludes the proof. The cases we left over are easier. Indeed, for $(p,q,r) = (1,1,1)$, it is a matter of pulling the $L^{\infty}$-norm of $\mathcal{P}(\tau_{v}\mathcal{R}\psi,1)(u)$ out of the integral (\ref{Young1.1}) using Theorem \ref{T3} with $(\infty,\infty,\infty)$, thus arriving at the constant
\begin{equation*}
 C = \left|S^{n-2}\right|\int^{1}_{-1}\dxi(s)\,.
\end{equation*}
In the case $(p,q,r) = (\infty, 1, \infty)$, it is just a matter of pulling out the $L^{\infty}$-norm of $f$ out of the integral (\ref{Young1.1}), and repeating the process as in integral $I_3$ above. The case $(p,q,r) = (1, \infty, \infty)$ is basically the same. In both these cases, the final constant can be read from (\ref{e17}). Note that one would obtain the same constants starting with the easy cases $(p,q,r) = (1,1,1), (1, \infty, \infty), (\infty,1,\infty)$ and using Riesz-Thorin interpolation, but we chose to do this directly.

In the case where $\alpha + \lambda >0$ and $\lambda>0$, we use two additional inequalities in order to control the post collisional local energy and powers of the post collisional velocity and of the the relative velocity by powers of the integration variables appearing on each $I_i, i=1,2,3$ in \eqref{add0}.  Indeed, from the energy dissipation we have
\begin{equation*}
|v'|^{2}+|v'_{*}|^{2}\leq |v|^{2}+|v_{*}|^{2}\,,
\end{equation*}
and thus
\begin{equation}\label{ine1}
|v-u^{-}|^{\alpha} = |v'|^{\alpha}\leq\left(|v|^{2}+|v_{*}|^{2}\right)^{\alpha/2} \leq 2^{\alpha/2}\left(|v|^{\alpha}+|v-u|^{\alpha}\right).
\end{equation}
Also,
\begin{equation}\label{ine2}
|u|^{\lambda}\leq \left(|v-u|+|v|\right)^{\lambda}\leq 2^{\lambda}\left(|v-u|^{\lambda}+|v|^{\lambda}\right).
\end{equation}
Now, take as a test function $\psi_{\alpha}(v)=\psi(v)|v|^{\alpha}$. The first step  is to use (\ref{e3}) and (\ref{ine2}) to obtain
\begin{align}\label{add1}
\begin{split}
\int_{\mathbb{R}^{n}}Q^{+}(f,g)&(v) \psi_{\alpha}(v)\,\dv =
 \int_{\mathbb{R}^{n}}\int_{\mathbb{R}^{n}}f(v)g(v-u)\mathcal{P}(\tau_{v}\mathcal{R}\psi_{\alpha},1)(u)\ |u|^{\lambda}\,\du\,\dv\\
&\leq 2^{\lambda}\Bigl\{\int_{\mathbb{R}^{n}}\int_{\mathbb{R}^{n}}f(v)g(v-u)\mathcal{P}(\tau_{v}\mathcal{R}\psi_{\alpha},1)(u)\ |v|^{\lambda}\,\du\,\dv\\
&  \ \ \ \ \ \ + \int_{\mathbb{R}^{n}}\int_{\mathbb{R}^{n}}f(v)g(v-u)\mathcal{P}(\tau_{v}\mathcal{R}\psi_{\alpha},1)(u)\ |v-u|^{\lambda}\,\du\,\dv \Bigr\}\,.
\end{split}
\end{align}
Making use of (\ref{P}) and (\ref{ine1}) yields
\begin{align}\label{add2}
\begin{split}
\mathcal{P}(\tau_{v}\mathcal{R}&\psi_{\alpha},1)(u) = \int_{S^{n-1}} \psi(v - u^{-}) |v-u^{-}|^{\alpha}b(\hat{u}\cdot\omega)\,\dom\,\\
& \leq 2^{\alpha/2} \Bigl\{\int_{S^{n-1}} \psi(v - u^{-}) |v|^{\alpha}b(\hat{u}\cdot\omega)\dom  \\
& \ \ \ \ \ \ \ \ \ \ \ \ \ \ \ \ \ \ + \int_{S^{n-1}} \psi(v - u^{-}) |v-u|^{\alpha}b(\hat{u}\cdot\omega)\dom\Bigr\}\,.
\end{split}
\end{align}
Recalling the notation $\psi_{\alpha}(v)=\psi(v)|v|^{\alpha}$, combining  (\ref{add1}) and (\ref{add2}) one obtains
\begin{align}\label{add3}
\begin{split}
\int_{\mathbb{R}^{n}}Q^{+}(f,g)&(v) \psi_{\alpha}(v)\,\dv \\
& \leq 2^{\lambda + \tfrac{\alpha}{2}}\Bigl\{\int_{\mathbb{R}^{n}}\int_{\mathbb{R}^{n}}f_{\lambda+\alpha}(v)g(v-u)\mathcal{P}(\tau_{v}\mathcal{R}\psi,1)(u)\,\du\,\dv\\
&  \ \ \ \ + \int_{\mathbb{R}^{n}}\int_{\mathbb{R}^{n}}f_{\lambda}(v)g_{\alpha}(v-u)\mathcal{P}(\tau_{v}\mathcal{R}\psi,1)(u)\,\du\,\dv\\
&  \ \ \ \ + \int_{\mathbb{R}^{n}}\int_{\mathbb{R}^{n}}f_{\alpha}(v)g_{\lambda}(v-u)\mathcal{P}(\tau_{v}\mathcal{R}\psi,1)(u)\,\du\,\dv\\
&  \ \ \ \ + \int_{\mathbb{R}^{n}}\int_{\mathbb{R}^{n}}f(v)g_{\lambda + \alpha}(v-u)\mathcal{P}(\tau_{v}\mathcal{R}\psi,1)(u)\,\du\,\dv\Bigr\}\,.
\end{split}
\end{align}
We now repeat the procedure for the case $\alpha = \lambda = 0$, breaking each of the 4 integrals of (\ref{add3}) into 3 parts according to (\ref{add0}). One should then estimate each piece using Theorem \ref{T3} and simple inequalities of the form
\begin{equation*}
\|f\|_{L^p(\R^n)} \leq \|f\|_{L^p_{\alpha+ \lambda}(\R^n)} \ \ \textrm{and} \ \ \|f_{\alpha}\|_{L^p(\R^n)} \leq \|f\|_{L^p_{\alpha+ \lambda}(\R^n)} \,.
\end{equation*}
At the end we arrive at
\begin{equation*}
\int_{\mathbb{R}^{n}}Q^{+}(f,g)(v) \psi_{\alpha}(v)\,\dv  \leq \, 2^{\lambda + 2 + \tfrac{\alpha}{2}}\,C\, \|f\|_{L^{p}_{\alpha+\lambda}(\mathbb{R}^{n})} \, \|g\|_{L^q_{\alpha+\lambda}(\mathbb{R}^{n})}\, \|\psi\|_{L^{r'}(\mathbb{R}^{n})}.
\end{equation*}
This proves that
\begin{equation*}
\left\|Q^{+}(f,g)(v)|v|^{\alpha}\right\|_{L^r(\R^n)} \leq   \,\, 2^{\lambda + 2+ \tfrac{\alpha}{2}}C\, \|f\|_{L^{p}_{\alpha+\lambda}(\mathbb{R}^{n})} \, \|g\|_{L^q_{\alpha+\lambda}(\mathbb{R}^{n})}.
\end{equation*}
A similar reasoning provides
\begin{equation*}
\left\|Q^{+}(f,g)(v)\right\|_{L^r(\R^n)} \leq   \, 2^{\lambda+1}\, C\, \|f\|_{L^{p}_{\alpha+\lambda}(\mathbb{R}^{n})} \, \|g\|_{L^q_{\alpha+\lambda}(\mathbb{R}^{n})}\,,
\end{equation*}
and finally
\begin{equation}\label{Youngfinal}
\left\|Q^{+}(f,g)(v)\right\|_{L^r_{\alpha}(\R^n)} \leq 2^{\lambda + 2+ \tfrac{\alpha}{2}+ \tfrac{1}{r}}\, C\, \|f\|_{L^{p}_{\alpha+\lambda}(\mathbb{R}^{n})} \, \|g\|_{L^q_{\alpha+\lambda}(\mathbb{R}^{n})}\,,
\end{equation}
with $C$ given in (\ref{e17}). The proof of Theorem 1 is now completed.


\section{Hardy-Littlewood-Sobolev Type Inequality for Soft Potentials}

In this section we study the collision operator for soft potentials and prove Theorem \ref{HLS}. We divide this proof in three parts, making use of the convolution symmetry of $Q^=(f,g)$.

\subsection{The case $r < q$} From (\ref{e3}) we have
\begin{align}
I:=\int_{\mathbb{R}^{n}}Q^{+}(f,g)(v)\,&\psi(v)\,\dv=\int_{\mathbb{R}^{n}}\int_{\mathbb{R}^{n}}f(v)g(v-u)\mathcal{P}(\tau_{v}\mathcal{R}\psi,1)(u)\, |u|^{\lambda}\,\du\,\dv \nonumber\\
&=\int_{\mathbb{R}^{n}}f(v)\left(\int_{\mathbb{R}^{n}}\tau_{v}\mathcal{R}g(u)\,\mathcal{P}(\tau_{v}\mathcal{R}\psi,1)(u)\ |u|^{\lambda}\du\right)\dv. \label{e19}
\end{align}
Applying H\"{o}lder's inequality and then Theorem \ref{T3} to the inner integral of (\ref{e19}), with $(p,q,r)=(a,\infty,a)$,   we obtain
\begin{align*}
\int_{\mathbb{R}^{n}}\tau_{v}\mathcal{R}g(u)\,\mathcal{P}(\tau_{v}\mathcal{R}\psi,& 1)(u)\ |u|^{\lambda}\du  \leq \ \left\|\mathcal{P}(\tau_{v}\mathcal{R}\psi,1)\right\|_{L^{a}(\mathbb{R}^{n},\,\dnull)}\left\|\tau_{v}\mathcal{R}g\right\|_{L^{a'}(\mathbb{R}^{n},\, \dnull)}\\
&\leq  C_1 \left\|\tau_{v}\mathcal{R}\psi\right\|_{L^{a}(\mathbb{R}^{n},\,\dnull)}\left\|\tau_{v}\mathcal{R}g\right\|_{L^{a'}(\mathbb{R}^{n},\, \dnull)}\\
&=C_1\Bigl[\bigl(|\psi|^{a}\ast|u|^{\lambda}\bigr)(v)\Bigr]^{1/a}\, \Bigl[\bigl(|g|^{a'}\ast|u|^{\lambda}\bigr)(v)\Bigr]^{1/a'}\, ,
\end{align*}
where $1/a + 1/a' = 1$ ($a$ to be chosen later), and the constant $C_1$ given by
\begin{equation}\label{constC1}
C_1 = \left|S^{n-2}\right|\, 2^{\tfrac{n+\lambda}{a}}\int^{1}_{-1}\left(\tfrac{1-s}{2}\right)^{-\tfrac{n+\lambda}{2a}}\,\dxi(s)\,.
\end{equation}
Therefore we arrive at
\begin{equation}\label{preHolder}
I \leq C_1 \int_{\mathbb{R}^{n}}f(v)\,\Bigl[\bigl(|\psi|^{a}\ast|u|^{\lambda}\bigr)(v)\Bigr]^{1/a}\, \Bigl[\bigl(|g|^{a'}\ast|u|^{\lambda}\bigr)(v)\Bigr]^{1/a'}\dv.
\end{equation}
Applying H\"{o}lder's inequality in (\ref{preHolder}) with exponents $1/p + 1/b + 1/c = 1$ ($b$ and $c$ to be chosen later) we arrive at
\begin{equation}\label{Holder4.3}
I \leq C_1\ \|f\|_{L^p(\R^n)}\ \left\||\psi|^{a}\ast|u|^{\lambda}\right\|_{L^{b/a}(\R^n)}^{1/a} \ \left\||g|^{a'}\ast|u|^{\lambda}\right\|_{L^{c/a'}(\R^n)}^{1/a'}
\end{equation}
We now use the classical Hardy-Littlewood-Sobolev inequality to obtain
\begin{equation}\label{HLS1}
\left\||\psi|^{a}\ast|u|^{\lambda}\right\|_{L^{b/a}(\R^n)} \leq C_2 \ \|\psi\|_{L^{ad}(\R^n)}^{a}
\end{equation}
and
\begin{equation}\label{HLS2}
\left\||g|^{a'}\ast|u|^{\lambda}\right\|_{L^{c/a'}(\R^n)} \leq C_3 \ \| g\|_{L^{a'e}(\R^n)}^{a'}\,,
\end{equation}
where
\begin{equation*}
1 +\frac{a}{b} = \frac{1}{d} - \frac{\lambda}{n} \ \ \ \textrm{and} \ \ \  1 +\frac{a'}{c} = \frac{1}{e} - \frac{\lambda}{n}\,.
\end{equation*}
The constants $C_2$ and $C_3$ (generally not sharp) are explicit in \cite[p. 106]{LiLo}. Finally putting together (\ref{HLS1}) and (\ref{HLS2}) with (\ref{Holder4.3}) we arrive at
\begin{equation}\label{FinalHLS}
I \leq C_1\, C_2^{1/a}\, C_3^{1/a'} \|f\|_{L^p(\R^n)}\,\| g\|_{L^{a'e}(\R^n)}\,\|\psi\|_{L^{ad}(\R^n)}.
\end{equation}
To conclude the proof of the theorem it would suffice to have in (\ref{FinalHLS}) the relations $a'e = q$ and $ad = r'$. Now it comes the moment to choose our variables. All the inequalities we used above will be well-posed if the following relations are satisfied
\begin{equation}\label{system-a}
\left\{
\begin{array}{ccc}
\dfrac{1}{a} + \dfrac{1}{a'} = 1, & 1 \leq a \leq \infty& \\
&&\\
\dfrac{1}{p} + \dfrac{1}{b} + \dfrac{1}{c} = 1, & 1 < b,c < \infty &\\
&&\\
1 + \dfrac{a}{b} = \dfrac{1}{d} - \dfrac{\lambda}{n}, & b>a, & 1<d<\infty\\
&&\\
1 + \dfrac{a'}{c} = \dfrac{1}{e} - \dfrac{\lambda}{n}, & c>a', & 1<e<\infty\\
&&\\
a'e = q &&\\
&&\\
ad = r'&&
\end{array}
\right.
\end{equation}
The last two equations determine $d$ and $e$ in terms of $a$. The remaining linear system (in the variables $1/a$, $1/a'$, $1/b$ and $1/c$) in undetermined because of the original relation
\begin{equation*}
 \frac{1}{p} + \frac{1}{q} = 1 + \frac{\lambda}{n} + \frac{1}{r}.
\end{equation*}
One can check that the choice
\begin{equation*}
\frac{1}{b} = \frac{1}{r'} - \frac{1}{a} \left( 1 + \frac{\lambda}{n}\right)
\end{equation*}
and
\begin{equation*}
\frac{1}{c} = \frac{1}{q} - \frac{1}{a'} \left( 1 + \frac{\lambda}{n}\right)
\end{equation*}
with any $1/a$ in the non-empty interval
\begin{equation}\label{goodinterval}
\textrm{max} \left\{ \frac{1}{r'}\,,\, 1 - \frac{1}{q( 1 + \tfrac{\lambda}{n})}\right\} < \frac{1}{a} < \textrm{min} \left\{ \frac{1}{r'(1 + \tfrac{\lambda}{n})}\,,\, \frac{1}{q'}\right\}
\end{equation}
provides a solution for \eqref{system-a}. If we define
\begin{equation}\label{boundb1*}
D_1 = C_1\, C_2^{1/a}\, C_3^{1/a'}
\end{equation}
using expressions (\ref{constC1}), (\ref{HLS1}), (\ref{HLS2}) and a choice of $a$ given by (\ref{goodinterval}), expression (\ref{FinalHLS}) is plainly equivalent to
\begin{equation}\label{boundb1}
\left\| Q^{+}(f,g)\right\|_{L^r(\mathbb{R}^{n})} \leq D_1 \,\|f\|_{L^{p}(\mathbb{R}^{n})} \, \|g\|_{L^q(\mathbb{R}^{n})},
\end{equation}
finishing the proof in this case.

\subsection{The case $r<p$} Here we make use of the convolution symmetry of $Q^{+}$ given by (\ref{e3}),
\begin{align}
I:=\int_{\mathbb{R}^{n}}Q^{+}(f,g)(v)\,&\psi(v)\,\dv=\int_{\mathbb{R}^{n}}\int_{\mathbb{R}^{n}}f(u+v)g(v)\mathcal{P}(1,\tau_{-v}\psi)(u)\, |u|^{\lambda}\,\du\,\dv \nonumber\\
&=\int_{\mathbb{R}^{n}}g(v)\left(\int_{\mathbb{R}^{n}}\tau_{-v}f(u)\,\mathcal{P}(1, \tau_{-v}\psi)(u)\ |u|^{\lambda}\du\right)\dv. \label{e19*}
\end{align}
We now proceed exactly as in the proof of the first case, noticing that the roles of $f$ and $g$ (and thus of $p$ and $q$) are interchanged. We will find
\begin{equation*}\label{FinalHLS*}
I \leq C_4\, C_5^{1/a}\, C_6^{1/a'} \|g\|_{L^q(\R^n)}\,\| f\|_{L^{a'e}(\R^n)}\,\|\psi\|_{L^{ad}(\R^n)}.
\end{equation*}
where $C_4$ (the analogous of $C_1$) is given by Theorem \ref{T3} as
\begin{equation*}
C_4=\left|S^{n-2}\right|\,\int^{1}_{-1}\left[\left(\tfrac{1+s}{2}\right)+(1-\beta_0)^{2}\left(\tfrac{1-s}{2}\right)\right]^{-\tfrac{n+\lambda}{2a}}\dxi(s)\,,
\end{equation*}
while $C_5$ and $C_6$ (the analogues of $C_2$ and $C_3$) are given by classical HLS inequalities as in (\ref{HLS1}) and (\ref{HLS2}), for a choice of $a$ now in the non-empty interval

\begin{equation}\label{goodinterval2}
\textrm{max} \left\{ \frac{1}{r'}\,,\, 1 - \frac{1}{p( 1 + \tfrac{\lambda}{n})}\right\} < \frac{1}{a} < \textrm{min} \left\{ \frac{1}{r'(1 + \tfrac{\lambda}{n})}\,,\, \frac{1}{p'}\right\}.
\end{equation}
In the end, defining
\begin{equation}\label{boundb2*}
D_2 = C_4\, C_5^{1/a}\, C_6^{1/a'}
\end{equation}
we will have
\begin{equation}\label{boundb2}
\left\| Q^{+}(f,g)\right\|_{L^r(\mathbb{R}^{n})} \leq D_2 \,\|f\|_{L^{p}(\mathbb{R}^{n})} \, \|g\|_{L^q(\mathbb{R}^{n})}.
\end{equation}

\subsection{The case $r \geq \max\{p,q\}$} The remaining range will be covered by the multilinear Riesz-Thorin interpolation \cite[Chapter 12, Theorem 3.3]{Z}. Recall that the triple $(p,q,r)$ satisfies $1 < p,q,r < \infty$ and
\begin{equation}\label{cond1aa}
 \frac{1}{p} + \frac{1}{q} = \left(1 + \frac{\lambda}{n}\right) + \frac{1}{r}.
\end{equation}
For fixed $\lambda$ and $n$, with $-n < \lambda <0$, there exist triples $(p_1, q_1, r)$, with $r<q_1$, and $(p_2,q_2,r)$, with $r<p_2$, such that:
\begin{equation}\label{cond2aa}
 \frac{1}{p_1} + \frac{1}{q_1} = \left(1 + \frac{\lambda}{n}\right) + \frac{1}{r} =  \frac{1}{p_2} + \frac{1}{q_2}.
\end{equation}
Therefore, for the triple $(p_1, q_1, r)$ we have bound (\ref{boundb1}) and for the triple $(p_2, q_2, r)$ we have bound (\ref{boundb2}). From (\ref{cond1aa}) and (\ref{cond2aa}) there is a $t$, with $0<t<1$, such that
\begin{equation*}
t\left(\frac{1}{p_1}, \frac{1}{q_1}, \frac{1}{r}\right) + (1-t)\left(\frac{1}{p_2}, \frac{1}{q_2}, \frac{1}{r}\right) = \left(\frac{1}{p}, \frac{1}{q}, \frac{1}{r}\right),
\end{equation*}
and we can apply the bilinear version of the Riesz-Thorin interpolation to get
\begin{equation}\label{boundb3}
\left\| Q^{+}(f,g)\right\|_{L^r(\mathbb{R}^{n})} \leq D_1^t D_2^{(1-t)} \,\|f\|_{L^{p}(\mathbb{R}^{n})} \, \|g\|_{L^q(\mathbb{R}^{n})}.
\end{equation}
This concludes the proof.

\section{Some immediate applications}
The radial symmetrization technique and the key ideas in the proofs of the basic inequalities for Boltzmann (Young's and HLS) described here have a deeper range and may be applied to other important estimates in kinetic theory. In this section we discuss two additional examples.
\subsection{Non-increasing Potentials}
Assume in the following discussion that the collision kernel has the form
\begin{equation}\label{dck}
B(|u|,\hat{u}\cdot\omega)=\Phi(u)\, b(\hat{u}\cdot\omega)\,,
\end{equation}
for some radial non-increasing potential $\Phi:\mathbb{R}^{n}\rightarrow\mathbb{R}^{+}$.  Let us retake the discussion of Section 3 by defining the measure $\nu_\Phi$ on $\R^n$ by
\begin{equation*}
\dnup=\Phi(x)\dx\,,
\end{equation*}
and the measure $\sigma^{\Phi}_{n}$ on $\R^{+}$ by
\begin{equation*}
\dsip(t)=\tilde{\Phi}(t)\;t^{n-1}\dt\,.
\end{equation*}
We claim that under this convention, Lemma \ref{L1} and Theorem \ref{T3} are still valid with these measures replacing $\dnua$ and $\dsig$ in the statement of these results. This follows from a simple observation in the proof of Lemma \ref{L1} when one performs the change of variables
\begin{equation*}
y=a_1(x,s) = \beta \,x \,\bigl(\tfrac{1-s}{2}\bigr)^{1/2}
\end{equation*}
Indeed, note that since $\Phi$ is non-increasing and $0\leq\beta\left(\tfrac{1-s}{2}\right)^{1/2}\leq1$, then
\begin{equation*}
\tilde{\Phi}(x)\leq\tilde{\Phi}(y).
\end{equation*}
Hence, using that $\beta \geq 1/2$, the estimate (\ref{eap1}) changes in this context to
\begin{equation*}
\left(\int^{\infty}_{0}\left|f(a_1(x,s))\right|^{p}\dsip(x)\right)^{\tfrac{1}{p}}\leq2^{\tfrac{n}{p}}\left(\tfrac{1-s}{2}\right)^{-\tfrac{n}{2p}}\left\|f\right\|_{L^p(\mathbb{R}^{+},\, \dsip)}.
\end{equation*}
Same observation is valid for the change of variables $y=a_2(x,s)$, in the sense that
\begin{equation*}
\left(\int^{\infty}_{0}\left|g(a_2(x,s))\right|^{q}\dsip(x)\right)^{\tfrac{1}{q}}\leq\left[\left(\tfrac{1+s}{2}\right)+(1-\beta_0)^{2}\left(\tfrac{1-s}{2}\right)\right]^{-\tfrac{n}{2q}}\left\|g\right\|_{L^q(\mathbb{R}^{+},\, \dsig)},
\end{equation*}
where $\beta_{0}=\beta(0)$.
We conclude that, for a radial non-increasing potential, Lemma \ref{L1} and Theorem \ref{T3} hold with the (non-sharp) constant
\begin{equation*}
C(n,p,q,b, \beta)=2^{-\tfrac{n}{p}}\int^{1}_{-1}\left(\tfrac{1-s}{2}\right)^{-\tfrac{n}{2p}}\left[\left(\tfrac{1+s}{2}\right)+(1-\beta_0)^{2}\left(\tfrac{1-s}{2}\right)\right]^{-\tfrac{n}{2q}}\dxi(s).
\end{equation*}

An immediate consequence of this observation is that Theorem \ref{HLS} holds for any collision kernel (\ref{dck}) with a potential that is radial non-increasing and belongs to a weak Lebesgue space $\Phi\in L^{s}_{weak}(\mathbb{R}^{n})$. A careful reading on the proof of this theorem together with the discussion above leads to:
\begin{corollary}[weakly integrable potential for the gain]\label{C7}
Let $1 < p,q,r,s <\infty$ with $1/p+1/q+1/s=1+1/r$. Assume that $\Phi\in L^{s}_{weak}(\mathbb{R}^{n})$ is a radial non-increasing function.  For the collision kernel {\rm(\ref{dck})} the bilinear operator $Q^{+}$ extends to a bounded operator from $L^{p}(\mathbb{R}^{n})\times L^{q}(\mathbb{R}^{n})\rightarrow L^{r}(\mathbb{R}^{n})$ via the estimate
\begin{equation*}
\left\|Q^{+}(f,g)\right\|_{L^{r}(\mathbb{R}^{n})}\leq C\left\|\Phi\right\|_{L^{s}_{weak}(\mathbb{R}^{n})}\left\|f\right\|_{L^{p}(\mathbb{R}^{n})}\left\|g\right\|_{L^{q}(\mathbb{R}^{n})},
\end{equation*}
with $C$ as in the proof of Theorem \ref{HLS}.
\end{corollary}

An interesting fact, that might be suitable for applications related to soft potentials, is that if we impose the stronger condition that our radial non-increasing potential belongs indeed to $L^s(\R^n)$, then we can include the endpoints in the range of exponents for our inequality. To see this, it is just a matter of following the proof of Theorem \ref{HLS} and use the classical Young's inequality for convolution whenever the classical HLS inequality was used (passages (\ref{HLS1}) and (\ref{HLS2})). One should obtain:

\begin{corollary}[integrable potential for the gain]\label{C8}
Let $1\leq p,q,r,s\leq\infty$ with $1/p+1/q+1/s=1+1/r$. Assume that $\Phi\in L^{s}(\mathbb{R}^{n})$ is a radial non-increasing function.  For the collision kernel {\rm (\ref{dck})} the bilinear operator $Q^{+}$ extends to a bounded operator from $L^{p}(\mathbb{R}^{n})\times L^{q}(\mathbb{R}^{n})\rightarrow L^{r}(\mathbb{R}^{n})$ via the estimate
\begin{equation*}
\left\|Q^{+}(f,g)\right\|_{L^{r}(\mathbb{R}^{n})}\leq C\left\|\Phi\right\|_{L^{s}(\mathbb{R}^{n})}\left\|f\right\|_{L^{p}(\mathbb{R}^{n})}\left\|g\right\|_{L^{q}(\mathbb{R}^{n})},
\end{equation*}
with $C$ as in the proof of Theorem \ref{HLS}.
\end{corollary}

\subsection{Estimates for the loss operator $Q^-(f,g)$}
This discussion leads to the application of the previous results to the loss operator $Q^-(f,g)$ defined in \eqref{defQ-} in strong form.
The following results assert that Theorem \ref{HLS} and its Corollaries \ref{C7} and \ref{C8} are also valid for this operator, however in a restricted range of exponents due to the fact that  this operator lacks the convolution symmetry of $Q^{+}$ given by (\ref{e3}) (see expression (\ref{nosymconv}) in the proof below).
We also show a counterexample that exhibits the lack of full range of exponents property that the gain operator $Q^+$ enjoys.

\begin{corollary}[weakly integrable potential for the loss]\label{C9}
Let $1 < p,q,r,s <\infty$ with $1/p+1/q+1/s=1+1/r$ and $r <p$. Assume that $\Phi\in L^{s}_{weak}(\mathbb{R}^{n})$ is a radial non-increasing function.  For the collision kernel {\rm(\ref{dck})} the bilinear operator $Q^{-}$ extends to a bounded operator from $L^{p}(\mathbb{R}^{n})\times L^{q}(\mathbb{R}^{n})\rightarrow L^{r}(\mathbb{R}^{n})$ via the estimate
\begin{equation*}
\left\|Q^{-}(f,g)\right\|_{L^{r}(\mathbb{R}^{n})}\leq C_5^{1/a}C_6^{1/a'}\left\|b\right\|_{L^{1}(S^{n-1})}  \left\|\Phi\right\|_{L^{s}_{weak}(\R^n)}\left\|f\right\|_{L^{p}(\mathbb{R}^{n})}\left\|g\right\|_{L^{q}(\mathbb{R}^{n})}.
\end{equation*}
with $C_5$, $C_6$ and $a$ as in the proof of Theorem \ref{HLS} (second case).
\end{corollary}

\begin{corollary}[integrable potential for the loss]\label{C10}
Let $1\leq p,q,r,s\leq\infty$ with $1/p+1/q+1/s=1+1/r$ and $r\leq p$. Assume that $\Phi\in L^{s}(\mathbb{R}^{n})$ is a radial non-increasing function.  For the collision kernel {\rm (\ref{dck})} the bilinear operator $Q^{-}$ extends to a bounded operator from $L^{p}(\mathbb{R}^{n})\times L^{q}(\mathbb{R}^{n})\rightarrow L^{r}(\mathbb{R}^{n})$ via the estimate
\begin{equation*}
\left\|Q^{-}(f,g)\right\|_{L^{r}(\mathbb{R}^{n})}\leq \left\|b\right\|_{L^{1}(S^{n-1})}\left\|\Phi\right\|_{L^{s}(\mathbb{R}^{n})}\left\|f\right\|_{L^{p}(\mathbb{R}^{n})}\left\|g\right\|_{L^{q}(\mathbb{R}^{n})}.
\end{equation*}

\end{corollary}

\begin{proof}
These estimates can be proved noticing that for any test function $\psi$ one has
\begin{align}
\begin{split}\label{nosymconv}
\int_{\mathbb{R}^n}Q^{-}(f,g)(v)\psi(v)\dv&=\left\|b\right\|_{L^{1}(S^{n-1})}\int_{\mathbb{R}^n}g(v)\left(\int_{\mathbb{R}^n}f(u)\psi(u)\Phi(v-u)\;\du\right)\dv\\
\leq\left\|b\right\|_{L^{1}(S^{n-1})}&\int_{\mathbb{R}^n}g(v)[(|\psi|^{a}\ast\Phi)(v)]^{1/a}[(|f|^{a'}\ast\Phi)(v)]^{1/a'}\dv,
\end{split}
\end{align}
where $1/a+1/a'=1$.  Now it is just a matter of following the proof of the second case of Theorem \ref{HLS}.  Note that estimates (\ref{HLS1}) and (\ref{HLS2}) are achieved using the classical version of HLS inequality when $\Phi$ is weakly integrable.  Meanwhile, when $\Phi$ is integrable they are achieved using the classical Young's inequality.  It is precisely in this step that the range of exponents differ in the two corollaries.
\end{proof}

{\bf Remark:} One should not expect Corollaries \ref{C9} and \ref{C10} to hold outside the ranges above.
For instance, in Corollary \ref{C9} we could examine the situation $p = q = r = s = 2$, taking the potential $\Phi(x) = |x|^{-n/2}$ and
\begin{equation*}
g(x) = \left\{
\begin{array}{ccc}
|x|^{-n/2}\;\left(\ln |x|\right)^{-1} & \textrm{if}\ \  |x| \geq e,\\
0 & \textrm{otherwise}.
\end{array}
\right.
 \end{equation*}
In this situation we would have for any $v$,
\begin{align*}
\int_{\mathbb{R}^{n}}g(v_*)\Phi(u)\dvs &= \int_{\mathbb{R}^{n}}g(x)\Phi(v-x)\dx\\& = \int_{\{|x|\geq e\}} |x|^{-n/2}\;(\ln |x|)^{-1}\;|x-v|^{-n/2}\,\dx\\ &\geq \frac{1}{2^{n/2}}\int_{\{|x|\geq \max\{e,|v|\}\}} |x|^{-n}(\log |x|)^{-1}\,\dx = \infty\, ,
\end{align*}
that makes $Q^{-}(f,g)(v)$ to blow up pointwise in (\ref{defQ-}), and thus the bound of Corollary \ref{C9} would not hold.

For Corollary \ref{C10} we could even consider a simpler example taking $\Phi(x)$ constant ($s=\infty$), $p=1$ and $q=r=\infty$. From (\ref{defQ-}) it is easy to see that the bound in Corollary \ref{C10} cannot hold in this case.

\subsection{Inequalities with exponential weights}

In this subsection we will present a Young-type estimate for the Boltzmann collision operator with stretched exponential weights.
Our motivation for presenting such estimates is the study of tail propagation, for instance, in solutions of the elastic Boltzmann equation, in self similar solutions for the inelastic Boltzmann equation, or, in solutions of the inelastic Boltzmann adding heating sources.

It is well known that the homogeneous elastic Boltzmann equation propagates solutions with Maxwellian tails.
More specifically, it was proved by Bobylev \cite{Bo-97} that such solutions  $f(t,v)$, in 3-dimensions and for hard spheres, satisfy the $L^{1}$-Maxwellian weighted estimate if initially so, that is,
$$\mbox{if}\ \ \left\|f_0(v)e^{a_0|v|^{2}}\right\|_{L^{1}}<\infty\ \ \mbox{then}\ \ \sup_{t\geq0}\left\|f(t,v)e^{a|v|^{2}}\right\|_{L^{1}}<\infty\ \ \mbox{for some}\ \ a\le a_0.$$
Later, a program based in a comparison principle was introduced in \cite{GPV}, for the $n$-dimensional case with variable potentials and integrable differential cross section, to obtain exponentially weighted pointwise estimates
\begin{equation*}
\sup_{t\geq0}\left\|f(t,v)e^{a|v|^{2}}\right\|_{L^{\infty}}<\infty,
\end{equation*}
under the assumption that $f_0$ satisfies similar $L^{\infty}$-Maxwellian integrability and not necessarily with same rate $a$.  One of the key ingredients in this program is the inequality (see Lemma 5 in \cite{GPV})
\begin{equation*}
\left\|Q^{+}(e^{-a|v|^{2}},f)(v)e^{a|v|^{2}}\right\|_{L^{\infty}}\leq C\left\|f(v)e^{a|v|^{2}}\right\|_{L_k^{1}}.
\end{equation*}
with $k=k(\lambda, n)$, with $k=0$ for the particular case of hard spheres in three dimension (i.e. $\lambda =1$ and $n=3$).
This estimate is a particular case of the Young's inequality with Maxwellian weights.  The techniques used in \cite{GPV}, and the subsequent applications to propagation of Maxwellian tails for derivatives \cite{AG08}, were based on the strong formulation of the collisional integral by the Carleman representation.

Further in the studies of inelastic Boltzmann equation, it was shown that solutions of the homogeneous inelastic Boltzmann for \textit{constant restitution coefficient} behave completely different from those of the elastic.  Thus, the previous $L^{\infty}$-Maxwellian estimate does not hold for them.  Indeed, as the time passes by, the gas cools down to complete rest.  In other words, the density converges to a Dirac distribution at $v=0$.  This result was first rigorously proved for the Maxwell type of interaction models by Bobylev, Carrillo and Gamba \cite{BCaG00}.  One may obtain such convergence by a dynamical rescaling that leaves the collisional integral invariant usually referred as a self-similar transformation.  It was actually shown that for the case of any dissipative model of Maxwell type of interactions, such transformation yield to a static frame where the distributions have power law for their high energy tails, \cite{BoCe02,BoCeTo03,BCG-09}.

However this picture is very different for hard potentials, where stretched exponential tails must be expected.  Indeed, Bobylev, Gamba and Panferov proved that any solution to the self-similar transformed static problem of the inelastic Boltzmann equation has solutions in $L^1$ with a weight that corresponds to a non Gaussian exponential high energy tail \cite{BGP}. Their studies also showed similar behavior to other inelastic collisional models under heating sources, where their stationary solutions are shown to have $L^{1}$-exponential weighted estimates with weights depending on the restitution coefficient and the heating mechanism.  Later, the time propagation of $L^{1}$-exponential tails by self-similarity, in the hard spheres case, was shown by Mischler, Mouhot and Ricart \cite{MMR}.

All these estimates, rigorously proved in the inelastic case for hard spheres and constant restitution coefficient, yield
\begin{equation*}
\sup_{t\geq0}\left\|f_{s}(t,v)e^{a|v|}\right\|_{L^{1}}<\infty
\end{equation*}
where $f_{s}$ is the self-similar profile. In addition, for hard potentials $0<\lambda<1$ one expects using formal arguments,
\begin{equation}\label{exp-int}
\sup_{t\geq0}\left\|f_{s}(t,v)e^{a|v|^{\lambda}}\right\|_{L^{1}}<\infty \,  .
\end{equation}
Consequently, a companion issue arises on the availability of exponentially weighted pointwise estimates for self-similar profiles $f_{s}$ associated to dissipative collisional equations, that is
\begin{equation}\label{exp-ine}
\sup_{t\geq0}\left\|f_{s}(t,v)e^{a|v|^{\lambda}}\right\|_{L^{\infty}}<\infty\, .
\end{equation}
Passing from \eqref{exp-int} to \eqref{exp-ine} will require an inequality of the type of \eqref{IMW1}.\\

Another source for motivation in searching $L^p$-exponentially weighted estimates
can be found in a forthcoming work \cite{Alon-Lods} where different issues of the self similar
asymptotic of inelastic collisions for variable restitution coefficient are treated.

\medskip

Motivated by the above discussion,  we present two Young-type estimates for the Boltzmann collision
operator with stretched exponential weights which preserve the rate of the exponential for elastic and inelastic interactions.
In order to shorten notation, fix $a>0$ and define for $\gamma\geq0$ the exponential weight and the Maxwellian weight respectively as
\begin{equation*}
\mathcal{M}_{a,\gamma}(v):=\exp\left(-a|v|^{\gamma}\right)\ \ \ \mbox{and}\ \ \ \mathcal{M}_{a}(v):=\mathcal{M}_{a,2}(v).
\end{equation*}
The first result of this section is an inequality which is ideal in the study of pointwise Maxwellian tails for hard potential
in the elastic case.  The proof that we give works for both elastic and inelastic cases.
\begin{proposition}\label{IMW}
Let $1 \leq p,\,q,\,r \leq \infty$ with $1/p + 1/q = 1 + 1/r$. Assume that
\begin{equation*}
B(|u|,\hat{u}\cdot\omega)=|u|^{\lambda}\, b(\hat{u}\cdot\omega)\,,
\end{equation*}
with $\lambda \geq 0$.  Then, for $a >0$,
\begin{equation}\label{IMW1}
\left\| Q^{+}(f,g)\ \mathcal{M}^{-1}_{a}\right\|_{L^r(\mathbb{R}^{n})} \leq C \,\|f\ \mathcal{M}^{-1}_{a}\|_{L^{p}(\mathbb{R}^{n})} \, \|g\ \mathcal{M}^{-1}_{a}\|_{L^q_{\lambda}(\mathbb{R}^{n})}.
\end{equation}
\end{proposition}
\begin{proof}
Using (\ref{e1}) and (\ref{eq1'}) we obtain
\begin{align}
I&:=\int_{\mathbb{R}^{n}}Q^{+}(f,g)(v)\left(\mathcal{M}^{-1}_{a}\psi\right)(v)\,\dv \nonumber\\
&=\int_{\mathbb{R}^{n}}\int_{\mathbb{R}^{n}}f(v)g(v-u)\int_{S^{n-1}}\left(\mathcal{M}^{-1}_{a}\psi\right)(v')\,|u|^{\lambda}\,b(\hat{u}\cdot\omega)\, \dom \,\du\,\dv.\label{Sec5.1}
\end{align}
From the energy dissipation we have $|v'|^{2}+|v'_{*}|^{2}\leq |v|^{2}+|v_{*}|^{2}$, and thus
\begin{equation}\label{energydiss}
\mathcal{M}^{-1}_{a}(v')\leq \mathcal{M}^{-1}_{a}(v)\,\mathcal{M}^{-1}_{a}(v_{*})\,\mathcal{M}_{a}(v'_{*}).
\end{equation}
Using (\ref{energydiss}) in (\ref{Sec5.1}) we obtain
\begin{align}\label{IMW2}
I&\leq\int_{\mathbb{R}^{n}}\int_{\mathbb{R}^{n}}\left(\mathcal{M}^{-1}_{a}f\right)(v)\left(\mathcal{M}^{-1}_{a}g\right)\!(v-u)\int_{S^{n-1}}\psi(v')\;\mathcal{M}_a(v'_{*})\,|u|^{\lambda}\,b(\hat{u}\cdot\omega)\, \dom \,\du\,\dv.
\end{align}
Recall that
\begin{equation*}
|u^{-}|= \beta|u|\sqrt{\tfrac{1-\hat{u}\cdot\omega}{2}}.
\end{equation*}
Therefore, using the fact that $\beta\geq 1/2$, we have the simple pointwise estimate,
\begin{equation*}
|u|^{\lambda}\leq 2^{\lambda}\left( \tfrac{2}{1-\hat{u}\cdot\omega} \right)^{\lambda/2}\;|u^{-}|^{\lambda}\leq 2^{2\lambda-1}\left( \tfrac{2}{1-\hat{u}\cdot\omega} \right)^{\lambda/2}\;\left( |v_*+u^{-}|^{\lambda}+|v_*|^{\lambda} \right).
\end{equation*}
Further, $v'_*=v_*+u^{-}$, then
\begin{equation}\label{Sec5.3}
\mathcal{M}_a(v'_*)\left( |v_*+u^{-}|^{\lambda}+|v_*|^{\lambda} \right)\leq C_{\lambda, a}\,(1+|v_*|^{\lambda}).
\end{equation}
Using \eqref{Sec5.3} we have
\begin{align}\label{Sec5.2}
\begin{split}
\int_{S^{n-1}}\psi(v')\;&\mathcal{M}_a(v'_{*})\,|u|^{\lambda}\,b(\hat{u}\cdot\omega)\, \dom \\
& \leq C_{\lambda,a}\left( 1 + |v_*|^{\lambda} \right)\int_{S^{n-2}}\psi(v-u^{-})\,\tilde{b}(\hat{u}\cdot{\omega})\, \dom,
\end{split}
\end{align}
where we have defined
\begin{equation*}
 \tilde{b}(\hat{u}\cdot{\omega}):=\left( \tfrac{2}{1-\hat{u}\cdot\omega} \right)^{\lambda/2}\;b(\hat{u}\cdot\omega).
\end{equation*}
Using \eqref{Sec5.2} in expression (\ref{IMW2}) we arrive at (recall that $v_*=v-u$)
\begin{equation*}
I\leq C_{\lambda,a}\,\int_{\mathbb{R}^{n}}\int_{\mathbb{R}^{n}}\left(\mathcal{M}^{-1}_{a}f\right)(v)\left(\mathcal{M}^{-1}_{a}g\right)(v-u)(1+|v-u|^{\lambda})\mathcal{P}(\tau_{v}\mathcal{R}\psi,1)(u)\,\du\,\dv.
\end{equation*}
We now have arrived at the same expression given in (\ref{Young1.1}), with $f(v)$ changed by $\left(\mathcal{M}^{-1}_{a}f\right)(v)$ and $g(v)$ changed by $\left(\mathcal{M}^{-1}_{a}g\right)(v)\left( 1+ |v|^{\lambda} \right)$. Repeating the argument for the Young's inequality in Section 3 we will conclude that
\begin{equation}
 \left\| Q^{+}(f,g)\ \mathcal{M}^{-1}_{a}\right\|_{L^r(\mathbb{R}^{n})} \leq C\; C_{\lambda,a}\,\|f\ \mathcal{M}^{-1}_{a}\|_{L^{p}(\mathbb{R}^{n})} \, \|g\ \mathcal{M}^{-1}_{a}\|_{L^q_{\lambda}(\mathbb{R}^{n})},
\end{equation}
with $C$ given by (\ref{e17}) with $\tilde{b}$ in place of $b$ and $C_{\lambda,a}$ from \eqref{Sec5.2}. This concludes the proof.
\end{proof}
If we restrict ourselves and impose ``strict dissipation conditions'' to the model we can do even better.
Indeed, note that  both norms in the right-hand side of estimate \eqref{IMW1-1} below are free from any extra weight coming from the potential.  This feature is achieved by imposing additional conditions (meant as ``strict dissipation'') on the restitution coefficient $e$ stated in the theorem.  Such conditions are not stringent since they are satisfied by standard models such as viscoelastic spheres or constant restitution coefficient $e<1$.
\begin{theorem}\label{ISW}
Let $1 \leq p,\,q,\,r \leq \infty$ with $1/p + 1/q = 1 + 1/r$. Assume that
\begin{equation*}
B(|u|,\hat{u}\cdot\omega)=|u|^{\lambda}\, b(\hat{u}\cdot\omega)\,,
\end{equation*}
with $0\leq\lambda\leq2$. Then, for non-increasing restitution coefficient such that $e(z)<1$ for $z\in(0,\infty)$,
\begin{equation}\label{IMW1-1}
\left\| Q^{+}(f,g)\ \mathcal{M}^{-1}_{a,\lambda}\right\|_{L^r(\mathbb{R}^{n})} \leq
C\,\|f\ \mathcal{M}^{-1}_{a,\lambda}\|_{L^{p}(\mathbb{R}^{n})} \,
\|g\ \mathcal{M}^{-1}_{a,\lambda}\|_{L^q(\mathbb{R}^{n})}.
\end{equation}
The constant $C:=C(n, \lambda, p,q, b,\beta)$ is computed below in the proof.  In the important
case $(p,q,r)=(\infty,1,\infty)$ this constant reduces to
\begin{equation}\label{cfinal}
C=C(n,\lambda)\int^{1}_{-1}\left[\left(\tfrac{1+s}{2}\right)+(1-\beta(0))^{2}\left(\tfrac{1-s}{2}\right)\right]^{-n/2}b_\beta(s)\,\ds,
\end{equation}
where
\begin{equation*}
 b_\beta(s):=\left[1-\left(\tfrac{1+|\vartheta_{\beta}(s)|}{2}\right)^{\lambda/2}\right]^{-1}b(s)\,,
\end{equation*}
with
\begin{equation*}
|\vartheta_{\beta}(s)|=\sqrt{(1-\beta(x))^{2}+\beta^{2}(x)+2(1-\beta(x))\beta(x)s}\,\,, \qquad \text{for}\ \
x=\sqrt{\frac{1-s}{2}}\, .
\end{equation*}
\end{theorem}
\begin{proof}
Let us introduce the classical center of mass-relative velocity change of coordinates
$$U=\frac{v+v_*}{2}\ \ \mbox{and}\  \  \  \vartheta =(1-\beta)\hat{u}+\beta\omega.$$
One can readily verify the following standard identities
\begin{equation*}
v'=U+\frac{|u|}{2}\vartheta\, \ \ \mbox{and}\ \ \ E:=|v|^{2}+|v_*|^{2}=2|U|^{2}+\frac{|u|^{2}}{2}.
\end{equation*}
Therefore, a direct calculation shows that
\begin{equation*}
|v'|^{2}=E\left(\tfrac{1+\xi\;\hat{U}\cdot\vartheta}{2}\right)\,, \qquad \text{with}\ \
\xi:=\frac{2|U||u|}{E}\leq 1\, .
\end{equation*}
Thus, for any test function $\psi$ with unitary $L^{r'}$-norm,
\begin{align*}
I&:=\int_{\mathbb{R}^{n}}Q^{+}(f,g)(v)\left(\mathcal{M}^{-1}_{a,\lambda}\psi\right)(v)\,\dv \nonumber\\
&=\int_{\mathbb{R}^{n}}\int_{\mathbb{R}^{n}}f(v)g(v-u)\int_{S^{n-1}}\left(\mathcal{M}^{-1}_{a,\lambda}\psi\right)(v')\,|u|^{\lambda}\,b(\hat{u}\cdot\omega)\, \dom \,\du\,\dv.
\end{align*}
In order to estimate $I$, we split the integral into two regions of integration: $A=\{|u|\leq1\}$ and its complement.
 The integral in the region $A$ is estimated by
\begin{equation}\label{2Sec5.1}
I_A\leq\int_{\mathbb{R}^{n}}\int_{\mathbb{R}^{n}}\left(\mathcal{M}^{-1}_{a,\lambda}f\right)(v)
\left(\mathcal{M}^{-1}_{a,\lambda}g\right)(v-u)\int_{S^{n-1}}\psi(v')\,b(\hat{u}\cdot\omega)\, \dom \,\du\,\dv.
\end{equation}
Indeed, this estimate follows from local dissipation of energy estimates (notice here the condition $0 \leq \lambda \leq 2$
is used), namely
\begin{equation*}
 |v'|^{\lambda}\leq \left(|v'|^{2}+|v'_*|^{2}\right)^{\lambda/2}\leq E^{\lambda/2}\leq|v|^{\lambda}+|v_*|^{\lambda} =|v|^{\lambda}+|v-u|^{\lambda}.
\end{equation*}
Hence, proceeding as in the proof of Young's type inequality in section 4, it follows
\begin{equation}\label{es1}
I_A\leq C_1\left\|f\mathcal{M}^{-1}_{a,\lambda}\right\|_{L^{p}(\mathbb{R}^{n})}\left\|g\mathcal{M}^{-1}_{a,\lambda}
\right\|_{L^{q}(\mathbb{R}^{n})} \, ,
\end{equation}
where the constant $C_1$ is explicitly computed by formula (\ref{e17}).

The second integral is slightly more involved.  In order to control the integral $I_{A^c}$ we use the inelastic interactions
law and conditions \eqref{beta} to our advantage.  First, observe that $\vartheta$ is a convex combination of two unitary vectors and since $\beta\geq 1/2$, the magnitude of $\vartheta$ increases as $\beta$ gets closer to $1$.  Moreover, $\beta$ is non-increasing, thus in $A^{c}=\{|u|\geq 1\}$ one has
\begin{multline*}
\left|\vartheta\left(\beta\left(|u|\sqrt{\tfrac{1-\hat{u}\cdot\omega}{2}}\right)\right)\right|\leq\left|\vartheta\left(\beta\left(\sqrt{\tfrac{1-\hat{u}\cdot\omega}{2}}\right)\right)\right|\\
\\
:=\left|\vartheta_\beta(\hat{u}\cdot\omega)\right|=\sqrt{(1-\beta(x))^{2}+\beta^{2}(x)+2(1-\beta(x))\beta(x)\hat{u}\cdot\omega},
\end{multline*}
with $x=\sqrt{\tfrac{1-\hat{u}\cdot\omega}{2}}$.\\
Next, note that by assumptions \eqref{beta} on the restitution coefficient $e$, the magnitude $|\vartheta_\beta|\neq1$ except for $\hat{u}\cdot\omega=1$.  Therefore, the magnitude of $v'$ is controlled by
\begin{equation*}
|v'|^{2}\leq E\left(\tfrac{1+|\vartheta|}{2}\right)\leq E\left(\tfrac{1+|\vartheta_\beta(s)|}{2}\right).
\end{equation*}
Also note that
\begin{equation*}
|u|^{\lambda}\leq 2^{\lambda/2}E^{\lambda/2},
\end{equation*}
and thus, for we can estimate the integral $I_{A^{c}}$ as follows,
\begin{multline*}
\hspace*{-.3cm}I_{A^{c}}\leq 2^{\lambda/2}\hspace*{-.15cm}\int_{_{\mathbb{R}^{2n}}}\hspace*{-.2cm} f(v)g(v-u)
\int_{_{S^{n-1}}}\hspace*{-.2cm}E^ {\lambda/2}\exp\left(aE^{\lambda/2}
\left(\tfrac{1+|\vartheta_\beta|}{2}\right)^{\lambda/2}\right)\psi(v')\,b(\hat{u}\cdot\omega)\, \dom \,\du\,\dv.
\end{multline*}
Finally, noting that
\begin{multline*}
E^{\lambda/2}\exp\left(aE^{\lambda/2}\left(\tfrac{1+|\vartheta_\beta|}{2}\right)^{\lambda/2}\right)\\
=\exp\left(a\;E^{\lambda/2}\right)\;E^{\lambda/2}\exp\left(-aE^{\lambda/2}\left(1-\left(\tfrac{1+|\vartheta_{\beta}|}{2}\right)^{\lambda/2}\right)\right)\\
\leq a^{-1}\sup_{x\geq0}\left\{xe^{-x}\right\}\left(1-\left(\tfrac{1+|\vartheta_\beta|}{2}\right)^{\lambda/2}
\right)^{-1}\exp\left(a\;E^{\lambda/2}\right)\, ,
\end{multline*}
one can estimate $I_{A^{c}}$ by
\begin{multline}\label{es2}
I_{A^{c}}\leq \\2^{\lambda/2}(a\;e)^{-1}\int_{\mathbb{R}^{2n}}\left(M^{-1}_{a,\lambda}f\right)(v)
\left(M^{-1}_{a,\lambda}g\right)(v-u)
\int_{S^{n-1}}\psi(v')\,b_\beta(\hat{u}\cdot\omega)\, \dom \,\du\,\dv\\
\leq C_2 2^{\lambda/2}(a\;e)^{-1}\left\|f\mathcal{M}^{-1}_{a,\lambda}\right\|_{L^{p}(\mathbb{R}^{n})}\left\|g\mathcal{M}^{-1}_{a,\lambda}
\right\|_{L^{q}(\mathbb{R}^{n})}\, ,
\end{multline}
for any  $0 \leq \lambda \leq 2$, where we have defined
\begin{equation*}
b_\beta(s):=\left(1-\left(\tfrac{1+|\vartheta_\beta(s)|}{2}\right)^{\lambda/2}\right)^{-1}\;b(s).
\end{equation*}
The constant $C_2$ is defined again by (\ref{e17}) with $b_\beta$ replacing $b$.  We obtain the final constant $C$ by adding the constants obtained in (\ref{es1}) and (\ref{es2}), namely, $C=C_1+C_2$.  Note that $b_\beta$ has a singularity at $s=1$, which in most cases of interest is at least of first order.  In the case $(p,q,r)=(\infty,1,\infty)$ the constant can be taken as showed in (\ref{cfinal}).
\end{proof}
{\bf Remark:}  One wonders then, if estimate \eqref{IMW1-1} is also valid in the elastic case, or perhaps, even the also useful (but weaker) estimate
\begin{equation*}
\left\| Q^{+}(f,g)\ \mathcal{M}^{-1}_{a,\lambda}\right\|_{L^r(\mathbb{R}^{n})} \leq
C\,\|f\ \mathcal{M}^{-1}_{a,\lambda}\|_{L^{p}(\mathbb{R}^{n})} \,
\|g\ \mathcal{M}^{-1}_{a,\lambda}\|_{L^q_\lambda(\mathbb{R}^{n})}.
\end{equation*}
We leave this open question to the reader.

\section*{Acknowledgments}
The authors thank Diogo Arsenio, William Beckner and Eric Carlen for very valuable discussions that much have improved this manuscript. This material is based upon work supported by the National Science Foundation under agreements No. DMS-0635607 (E. Carneiro), DMS-0636586 and DMS-0807712 (R. Alonso and I. M. Gamba).
 E. Carneiro would also like to acknowledge support from the CAPES/FULBRIGHT grant  BEX 1710-04-4 and the Homer
Lindsey Bruce Fellowship from the University of Texas. Support from the Institute
from Computational Engineering and Sciences at the University
of Texas at Austin is also gratefully acknowledged.

\end{document}